\documentclass[12pt,a4paper,twoside,reqno]{amsart}
\usepackage[a4paper,top=3cm,bottom=3cm,inner=2.5cm,outer=2.5cm]{geometry}
\usepackage[utf8]{inputenc}
\usepackage{amsmath,amssymb,amsthm,amsfonts}
\usepackage[british]{babel}
\usepackage{xparse}
\usepackage{tikz}
\usetikzlibrary{calc,intersections}
\usepackage{tikz-cd}
\usepackage{multirow,enumitem,array,longtable}
\usepackage{pst-node,placeins,xspace,fix-cm}
\usepackage{nicefrac}
\usepackage{accents}
\usepackage{url}
\usepackage[nomessages]{fp}
\usepackage[colorlinks,linkcolor=blue,citecolor=violet,urlcolor=darkgray,final,hyperindex,linktoc=page,hyperfootnotes=true]{hyperref}
\newlength\horspace
\newlength\verspace
\renewcommand{\v}[1][1.0]{\vspace*{#1\verspace}\xspace}

\newtheorem{teor}{Theorem}[section]
\newtheorem{prop}[teor]{Proposition}
\newtheorem{coroll}[teor]{Corollary}
\newtheorem{lemm}[teor]{Lemma}

\newtheorem{prop.}[teor]{Proposition}

\theoremstyle{definition}
\newtheorem{defne}[teor]{Definition}
\newtheorem{rem}[teor]{Remark}
\newtheorem{exampl}[teor]{Example}

\newtheorem{exam.}[teor]{Example}
\newtheorem{def.}[teor]{Definition}

\def\nameit#1{\textrm{#1}~}
\def\thex{\nameit{Theorem}}
\def\prox{\nameit{Proposition}}
\def\corx{\nameit{Corollary}}
\def\lemx{\nameit{Lemma}}

\def\dfn#1{{\itshape #1}}

\newenvironment{enumT}{\begin{enumerate}[itemsep=2.5mm,label=$($\hspace{-0.1ex}\roman*\hspace{0.13ex}$)$]}{\end{enumerate}}
\DeclareFontFamily{OT1}{pzc}{}
\DeclareFontShape{OT1}{pzc}{m}{it}{<->s*[1.19]pzcmi7t}{}
\DeclareMathAlphabet{\mathpzc}{OT1}{pzc}{m}{it}
\newcommand{\catfont}[1]{\ensuremath{\mathpzc{#1}}\xspace}

\newcommand{\Open}{{\rm Open}}

\renewcommand{\o}[2][]{\catfont{o}_{#1}(#2)}
\newcommand{\cl}[2][]{\catfont{c}_{#1}(#2)}
\newcommand{\st}{^{\ast}}
\newcommand{\void}{\emptyset}
\newcommand{\cont}{\subseteq}
\newcommand{\contain}{\supseteq}
\newcommand{\inters}{\cap}
\newcommand{\meet}{\wedge}
\newcommand{\join}{\vee}

\newcommand{\bigjoin}{\bigvee}
\newcommand{\difference}{\setminus}
\newcommand{\closure}[2][]{\overline{#2}^{#1}}
\newcommand{\bd}[1]{\operatorname{bd}\left(#1\right)}
\newcommand{\interior}[1]{\operatorname{Int}\left(#1\right)}
\def\set#1#2{\left\{{#1}\left.\right|\,{#2}\right\}}
\newcommand{\fami}[2]{\left\{{#1}\right\}_{#2}}
\renewcommand{\Roman}[1]{\uppercase\expandafter{\romannumeral #1\relax}}
\makeatletter
\newcommand{\aR}[2][]{%
	\ext@arrow 0055{\Rightarrowfill@}{#1}{#2}}
\makeatother
\newcommand{\implic}{\aR{}}
\def\iff{\Leftrightarrow}
\newcommand{\Rom}[1]{\ensuremath{(\Roman{#1})}}
\newcommand{\Romprime}[1]{\ensuremath{(\Roman{#1}')}}
\newcommand{\Romplus}[1]{\ensuremath{(\Roman{#1}^{+})}}

\begin{document}

\title[Unicoherence in Locales]{Unicoherence in Locales}
\author[E.~Caviglia]{Elena~Caviglia}
\address{Stellenbosch University, Department of Mathematical Sciences, Stellenbosch, South Africa. And National Institute for Theoretical and Computational Sciences (NITheCS), South Africa}
\email{elena.caviglia@outlook.com}
\author[L.~Mesiti]{Luca~Mesiti}
\address{University of KwaZulu-Natal, School of Agriculture and Science, Discipline of Mathematics, Private Bag X54001, Durban 4000, South Africa}
\email{luca.mesiti@outlook.com}
\author[C.~Rathilal]{Cerene~Rathilal}
\address{University of KwaZulu-Natal, School of Agriculture and Science, Discipline of Mathematics, Private Bag X54001, Durban 4000, South Africa. and National Institute for Theoretical and Computational Sciences (NITheCS), South Africa}
\email{rathilalc@ukzn.ac.za}

\begin{abstract}
	In this paper, we generalize the concept of unicoherence to the context of frames. Unicoherence, originally introduced by Kuratowski, is a connectedness property that is well studied in classical topology and used to detect holes of a space. We extend the notion of unicoherence to locales and we then investigate its properties. In particular, we prove that many of the known characterizations of unicoherence for topological spaces extend to the setting of locales. Some of these characterizations interestingly involve separation properties for locales.
\end{abstract}

\maketitle

\setcounter{tocdepth}{1}
\tableofcontents

\section{Introduction}

Unicoherence is a connectedness property which was first defined by Kuratowski \cite{Kur26} in 1926. Intuitively, a space is unicoherent if it is connected and does not contain any holes. So for example all euclidean spaces are unicoherent while the circle is not. Kuratowski introduced this topological notion to better understand and characterize the sphere $S^2$. Indeed, the sphere is unicoherent although the circle is not. Later, unicoherent topological spaces have been extensively studied, by (but not limited to) Stone \cite{Sto49}, Wallace \cite{Wallace}, Eilenberg \cite{Eil36}, Borsuk (\cite{Bor31}, \cite{Bor33}, \cite{Borsuk33}) and more recently Dickman and Rubin \cite{DickmanRub72}. These investigations led to numerous interesting equivalent characterizations of unicoherence for topological spaces, many of which are collected in \cite{GarIll89}. Among such characterizations, Brouwer property is of particular relevance, capturing unicoherence in terms of a separation condition for closed subsets. In this paper, we extend the notion of unicoherence and many of its known equivalent characterizations to the setting of locales.

The development of pointfree topology, or locale theory, offers an algebraic framework that broadens topological ideas without directly referring to points \cite{topwop}. With this approach, one works with the lattice of open sets of a topological space rather than the points themselves. The objects of study in pointfree topology are called \emph{frames} or \emph{locales}. Many key topological concepts such as compactness, connectedness and separation properties have been successfully studied via locales (see Johnstone \cite{stonespaces},  Pultr and Picado \cite{topwop}). The study of connectedness in locales has created a rich parallel theory that reflects the classical setting while revealing constructive proofs.

Our work extends unicoherence to locales and studies its properties, generalizing classical results and developing new tools for pointfree topology. We present several equivalent characterizations of unicoherence for locales, exploring how unicoherence relates to sublocales, boundary operations, the components of open sublocales and separation conditions. 
 
 To present this work, we begin by recalling relevant background material on point-free topology. For details not provided here, we refer the reader to \cite{bannotes}, \cite{stonespaces} and \cite{topwop}.\\

A \textit{frame or locale} $L$ is a complete lattice which satisfies the infinite distributive law: 
\begin{center}$x \wedge \bigvee S = \bigvee \{ x \wedge s | s \in S\}$, \mbox{\;\;\;} for all $x \in L, S \subseteq L.$\end{center}The top element of $L$ is denoted by $1_L$ and the bottom by $0_L$.

If $X$ is a topological space, then $\mathcal{O}X = \{ U \subseteq X |\; U \mbox{ is open}\}$ is a frame with partial order $\subseteq$ and intersection as a binary meet. We call $\mathcal{O}X$ the \emph{frame of open sets} of $X$.

The \emph{pseudocomplement} of $a \in L$ is denoted $a^*$ and is characterized by the following formula $$a^* =\bigvee \{ x \in L \;| \; a \wedge x = 0\}.$$

 An element $x$ in a frame $L$ is said to be connected if whenever $x = b \vee c$ with $b\wedge c = 0_L$ we have either $b=0_L$ or $c=0_L$. A frame $L$ is connected if its top element $1_L$ is connected and  is  locally connected provided each element in the frame can be written as the join of connected elements.

If $L$ is a frame and $S \subseteq L$, then $S$ is called a \emph{sublocale} if $S$ is closed under arbitrary meets and if for $x \in L$ and $s \in S$, $(x \rightarrow s) \in S$. We denote the bottom element of $S$ by $0_S$, where $0_S =\bigwedge S$. The \emph{trivial} (also known as \emph{void}) sublocale is $\{1_L\}$ which we shall denote by $\emptyset.$ An arbitrary intersection (or meet) of sublocales, is a sublocale. The collection of all sublocales of a frame $L$, denoted by $\mathcal{S}(L)$, is a complete lattice $(\mathcal{S}(L), \subseteq)$.
For $S_i \in \mathcal{S}(L)$ for $i \in I$, we recall that the join of sublocales is defined as follows:
$$ \bigvee_{i \in I} S_i = \{ \bigwedge M \; | \; M \subseteq  \bigcup_{i \in I} S_i \}.$$
We note that the join of sublocales is indeed a sublocale. The \emph{open sublocale} associated with any $a \in L$ is : $ \mathfrak{o}_L(a) = \{ x \in L \;| \; a \rightarrow x = x \}$.
The \emph{closed sublocale} associated with any $a \in L$ is : $\mathfrak{c}_L(a) =\; \uparrow a = \{ x \in L \;| \; a \le x\}$. In the absence of ambiguity, we denote the open sublocale of $a$ in $L$ by $ \mathfrak{o}(a)$ and the closed sublocale of $a$ in $L$ by $\mathfrak{c}(a)$.
The \emph{closure} of a sublocale $S$ in $L$, denoted $\overline{S}_L$ (or $\overline{S}$ for simplicity), is defined as $\overline{S} = \;  \uparrow \left(\bigwedge S\right).$  We also note that if $S$ is connected, then $\overline{S}$ is connected. The  \emph{supplement} of $S$, denoted  $\sup(S)$, is given by  $\sup(S) = \bigwedge\{ T\; | \; T \mbox{ is a sublocale of } L,\; T \vee S = L \}$, and  $S$ is called \emph{complemented} sublocale in $L$ if, $S \wedge \sup(S) = \{1_L\} (= \emptyset)$. The \emph{interior} of $S$, denoted by $\interior{S}$, is $\interior{S} = \bigvee \{ U \in \mathcal{S}(L) |\; U \mbox{ is open and } U \subseteq S\}$ and correspondingly the \emph{boundary} of a sublocale $S$, denoted $\bd{S}$, is $\bd{S} = \overline{S}\setminus \interior{S}.$


\section{Properties of Sublocales}

In this section, we recall and outline several properties of sublocales. These properties will be essential for our subsequent study of connectedness and unicoherence in locales. Amongst the various properties, we will investigate the behavior of open and closed sublocales, along with their interaction under closure, interior, and boundary operations.

\begin{defne}[\cite{topwop}]
    A locale $L$ is \dfn{normal} if for every $a,b\in L$ such that $a\join b=1$ there exist $u,v\in L$ such that $u\meet v=0$, $a\join u=1$ and $b\join v=1$.

    Equivalently, $L$ is normal if for every $\cl{a},\cl{b}\cont L$ such that $\cl{a}\inters \cl{b}=\void$ there exist $\o{u},\o{v}\cont L$ such that $\o{u}\inters \o{v}=\void$, $\cl{a}\cont \o{u}$ and $\cl{b}\cont \o{v}$.
\end{defne}

We now present proofs of well-known facts about sublocales that are not provided in the literature. These facts will be used in our later results.

\begin{prop} \label{closedofsub}
	Let $S\cont L$ be a sublocale. A closed sublocale $\cl[S]{u}$ of $S$ is the same as $\cl[S]{u}\inters S$. An open sublocale $\o[S]{u}$ of $S$ is the same as $\o{u}\inters S$.

	Let $T\cont S$ be a sublocale. Then the closure $\closure[S]{T}$ of $T$ in $S$ is equal to $\closure{T}\inters S$.
\end{prop}
\begin{proof}
    By definition, given $u\in S$,
    $$\cl[S]{u}=\set{x\in S}{u\leq x}=\set{x\in L}{u\leq x}\inters S=\cl{u}\inters S.$$
    Then $\o[S]{u}$ is the unique complement in $S$ of $\cl[S]{u}=\cl{u}\inters S$. But $\o{u}\inters S$ is also a complement of $\cl{u}\inters S$ in $S$. Whence $\o[S]{u}=\o{u}\inters S$. 

    Consider now $T\cont S$. The closure $\closure[S]{T}$ is the smallest closed sublocale of $S$ containing $T$. Since by the argument above $\closure{T}\inters S$ is closed in $S$ and contains $T$, we have $\closure[S]{T}\cont \closure{T}\inters S$. Notice then that, by the argument above, $\closure[S]{T}$ needs to be of the form $\cl{u}\inters S$, with $T\cont \cl{u}$. But $\closure{T}$ is the smallest closed sublocale of $L$ containing $T$, so $\closure{T}\cont \cl{u}$. Whence $\closure{T}\inters S\cont \cl{u}\inters S=\closure[S]{T}$.
\end{proof}

\begin{prop}\label{closureintopen}
	Let $S\cont L$ be a sublocale and let $\o{u}\cont L$ be an open sublocale. If $\o{u}\inters \closure{S}\neq \void$ then also  $\o{u}\inters {S}\neq \void$. 
\end{prop}
\begin{proof}
We prove that if $\o{u} \inters S=\void$ then $\o{u}\inters \closure{S}= \void$. If $\o{u} \inters S=\void$, then $S \cont \cl{u}$ and so $S \cont \closure{S} \inters \cl{u}$. By definition of closure, this implies $\closure{S}= \closure{S} \inters \cl{u}$ and hence $\o{u}\inters \closure{S}= \void$.
\end{proof}

\begin{lemm} \label{neighinterior}
Let $S\cont L$ be a sublocale and let $x\in L$ with $x\neq 1$. Then $x\in \interior{S}$ if and only if there exists an open sublocale $\o{u}$ such that $x\in\o{u}$ and $\o{u} \cont S$.
\end{lemm}
\begin{proof}
If $x\in \interior{S}$, then $\interior{S}$ is an open sublocale contained in $S$ that contains $x$. If there exists an open sublocale $\o{u}$ such that $x\in\o{u}$ and $\o{u} \cont S$, then $x\in \interior{S}$ since  $ \interior{S}$ is the biggest open sublocale contained in $S$.
\end{proof}

\begin{lemm}\label{neighboundary}
Let $S\cont L$ be a complemented sublocale and let $x\in \bd{S}$ with $x\neq 1$. Then every open sublocale $\o{u}$ such that $x\in \o{u}$ is such that $\o{u} \inters S \neq \void$ and $\o{u} \inters (L\difference S)\neq \void$.
\end{lemm}
\begin{proof}
Let $\o{u}$ be an open sublocale of $L$ such that $x\in \o{u}$. Since $x\in\bd{S}\cont \closure{S}$, by Proposition \ref{closureintopen}, we have $\o{u} \inters S \neq \void$. Moreover, since $x\notin \interior{S}$, by Lemma \ref{neighinterior}, $\o{u} \inters (L\difference S)\neq \void$.
\end{proof}

\begin{lemm}\label{boundarydiff}
Let $C\cont L$ be a complemented sublocale of $L$. Then $\bd{C}=\bd{L\difference C}$.
\end{lemm}

\begin{proof}
To prove that $\bd{C}\cont\bd{L\difference C}$, it suffices to show that $\bd{C} \inters (L \difference \bd{L\difference C})=\void$. Suppose $\bd{C} \inters (L \difference \bd{L\difference C})\neq\void$. Since $\bd{C} \inters (L\difference \bd{L\difference C})=\bd{C} \inters ((L\difference (\closure{L \difference C})) \join \interior{L\difference C})= (\bd{C} \inters L \difference ( \closure{L \difference C})) \join (\bd{C} \inters \interior {L \difference C})$ it must be either $(\bd{C} \inters L \difference ( \closure{L \difference C})) \neq \void$ or $(\bd{C} \inters \interior {L \difference C})\neq \void$. But $(\bd{C} \inters L \difference ( \closure{L \difference C})) = \void$ because otherwise, by Lemma \ref{neighboundary} and Lemma \ref{neighinterior}, there would exist $x\in (\bd{C} \inters L \difference ( \closure{L \difference C}))$ and $U_x\cont L\difference (\closure{L\difference C})$ open such that $U_x \inters L \difference C \neq \void$ and this is a contradiction. Moreover, $(\bd{C} \inters \interior {L \difference C})= \void$ since otherwise, by Lemma \ref{neighboundary} and Lemma \ref{neighinterior}, there would exist $x\in (\bd{C} \inters \interior {L \difference C})$ and $U_x \cont L\difference C$ such that $U_x\inters C\neq \void$ and this is a contradiction. Hence it must be $\bd{C} \inters (L \difference \bd{L\difference C})=\void$, that implies $\bd{C}\cont\bd{L\difference C}$. The same argument for $L\difference C$ in place of $C$ shows that $\bd{L\difference C} \cont \bd{C}$.
\end{proof}

\begin{lemm}\label{bdinters}
    Let $A,B\cont L$ be sublocales of $L$. Then
    $$\bd{A\inters B}\cont \bd{A}\join \bd{B}$$
    $$\bd{A\join B}\cont \bd{A}\join \bd{B}$$

\end{lemm}
\begin{proof}
We notice that $\bd{A} \join \bd{B}=(\closure{A} \join \closure{B}) \inters (\closure{A} \join L\difference \interior{B}) \inters (L\difference \interior{A} \join \closure{B}) \inters (L\difference \interior{A}\join L\difference \interior{B})$. Moreover, $\bd{A \inters B} \cont \closure{A\inters B}$ and so $\bd{A \inters B} \cont \closure{A}$ and $\bd{A \inters B} \cont \closure{B}$. In addition to this, since $\interior{A} \inters \interior{B} \cont \interior{A \inters B}$, we have that $ L \difference \interior{A \inters B} \cont (L \difference \interior{A}) \join (L \difference \interior{B})$. Hence $\bd{A\inters B}$ is contained in $\bd{A} \join \bd{B}$.

Notice then that $\bd{A\join B}\cont \closure{A\join B}\cont \closure{A}\join \closure{B}$. Moreover, $\interior{A}\cont \interior{A\join B}$ and $\interior{B}\cont \interior{A\join B}$, whence $L\difference \interior{A\join B}\cont (L\difference \interior{A})\inters(L\difference \interior{B})$. We thus conclude that $\bd{A\join B}\cont \bd{A}\join \bd{B}$.
\end{proof}

\section{On Connectedness}

We now explore connected sublocales and understand their properties with respect to \emph{separated} sublocales, \emph{complemented} sublocales and their \emph{components}. We introduce the notion of \emph{strongly locally connected} locale, which in the context of locales differs from local connectedness.

\begin{defne}
    Let $P,Q\cont L$ be sublocales of $L$. $P$ and $Q$ are \dfn{separated} if $\closure{P}\inters Q=\void=P\inters \closure{Q}$.
\end{defne}

\begin{prop}\label{connectedchar}
	Let $S\cont L$ be a sublocale. The following are equivalent:
\begin{enumT}
\item $S$ is connected;
\item The only clopen sublocales of $S$ are $\void$ and $S$ itself;
\item If $S=P\join Q$ with $P$ and $Q$ separated, then $P=\void$ or $Q=\void$.
\end{enumT}
\end{prop}
\begin{proof}
We prove (i) $\iff$ (ii) and (i) $\iff$ (iii).

(i) $\Rightarrow$ (ii). If  $T$ is a clopen sublocale of $S$ such that $T\neq S$ and $T \neq \void$, then $T$ and $S\difference T$ exhibit a separation of $S$.

(ii) $\Rightarrow$ (i). If $S$ is not connected then there exists open sublocales $U$ and $V$ such that $S \cont U \join V$, $S \inters U \inters V\neq \void$, $S \inters U \neq \void$ and $S \inters V \neq \void$. Then $S \inters U$ and $S\inters V$ are non-void clopens of $S$ that are not the entire $S$, contradicting (ii).

 (i) $\Rightarrow$ (iii). Let $S=P\join Q$ with $P$ and $Q$ separated. Then $S \cont \closure{P} \join \closure{Q}$ and $S \inters \closure{P} \inters \closure{Q}= \void$. Then, since $S$ is connected, it must be either $S \inters \closure{P} =\void$ or $S \inters \closure{Q} =\void$ and hence $P= \void$ or $Q=\void$.

(iii) $\Rightarrow$ (i). Let $S\cont F_1 \join F_2$ with $F_1$ and $F_2$ closed sublocales. Then $F_1\inters A$ and $F_2 \inters A$ are separated and $S=(F_1\inters A) \join (F_2\inters A)$. Then by (iii) we have that $F_1 \inters A= \void$ or $F_2 \inters A= \void$ and hence $S$ is connected.
\end{proof}

\begin{lemm}\label{lemmadiffsep}
    Let $L$ be a connected locale and $C\cont L$ a connected and complemented sublocale of $L$. If $L\difference C=P\join Q$ with $P$ and $Q$ separated, then both $C\join P$ and $C\join Q$ are connected.
\end{lemm}
\begin{proof}
Let $A$ and $B$ be separated sublocales of $L$ such that $C \join P=A \join B$. Since $C\cont A\join B$ is connected, we have $C\inters A=\void$ or $C \inters B=\void$. Without loss of generality we suppose $C\inters A=\void$. Then $A \cont P$ and so $A$ and $Q$ are separated. Moreover $L=C \join P \join Q= A \join (B\join Q)$ and $A$ and $B\join Q$ are separated because $A$ is separated from both $B$ and $Q$. Since $L$ is connected this implies $A= \void$ or $B\join Q=\void$ and hence $B=\void$. This shows that $C\join P$ is connected.  The same argument can be used to show that $C\join Q$ is connected as well.
\end{proof}

\begin{prop}\label{connectedsub}
	Let $E\cont S$ be sublocales of $L$. Then $E$ is connected as a sublocale of $S$ precisely when it is connected as a sublocale of $L$.
\end{prop}
\begin{proof}
We  first prove that if $E$ is connected in $L$ then it is connected in $S$. Recall that, by Proposition \ref{closedofsub}, the closed sublocales of $S$ are of the form $\cl[S]{u}= \cl{u} \inters S$. Suppose now $E \cont \cl[S]{u} \join \cl[S]{v}$ with $E \inters \cl[S]{u} \inters \cl[S]{v}= \void$. Then $E\cont \cl{u} \inters \cl{v}$ and we have $E \inters \cl[]{u} \inters \cl[]{v}=E \inters \cl[S]{u} \inters \cl[S]{v}= \void$. Since $E$ is connected in $L$ this implies either $E\inters \cl{u}=\void$ or  $E\inters \cl{v}=\void$ and hence either  $E\inters \cl[S]{u}=\void$ or  $E\inters \cl[S]{v}=\void$.

We now prove that if $E$ is connected in $S$ it is connected in $L$ as well. Let $E\cont \cl{u} \join \cl{v}$ with $E\inters \cl{u} \inters \cl{v}=\void$. Then $E\cont \cl[S]{u} \join \cl[S]{v}$ and $E \inters \cl[S]{u} \inters \cl[S]{v}=E \inters \cl[]{u} \inters \cl[]{v}= \void$. Since $E$ is connected in $S$ this implies either $E\inters \cl[S]{u}=\void$ or  $E\inters \cl[S]{v}=\void$ and hence, since $E \inters S=E$, either $E\inters \cl[]{u}=\void$ or  $E\inters \cl[]{v}=\void$.

\end{proof}

In the next result, we present an intuitive characterisation of closed sublocales in terms of closed sublocales.
\begin{lemm}\label{closedsubconnected}
	For a closed sublocale $\cl{u}\cont L$ the following are equivalent:
	\begin{enumT}
		\item $\cl{u}$ is connected;
		\item Whenever $\cl{u}=\cl{a}\join \cl{b}$ for some $a,b\in L$ with $\cl{a}\inters\cl{b}=\void$, we have that $\cl{a}=\void$ or $\cl{b}=\void$.
	\end{enumT}
\end{lemm}
\begin{proof}
(i) $\implic$ (ii). Trivial using the definition of connected sublocale.

(ii) $\implic$ (i). Let $\cl{u}\cont \cl{a} \join \cl{b}$ with $\cl{u}\inters \cl{a} \inters \cl{b} =\void$. Consider then the disjoint closed sublocales $\cl{a} \inters \cl{u}$ and $\cl{b} \inters \cl{u}$. We have $\cl{u}=(\cl{a} \inters \cl{u})\join (\cl{b} \inters \cl{u})$ and $(\cl{a} \inters \cl{u})\inters (\cl{b} \inters \cl{u})= \void$. Hence, by (ii), we conclude that either  and hence $\cl{u}$ is connected.
(ii) $\implic$ (i). Let $\cl{u}\cont \cl{a} \join \cl{b}$ with $\cl{u}\inters \cl{a} \inters \cl{b} =\void$. Consider then the disjoint closed sublocales $\cl{a} \inters \cl{u}$ and $\cl{b} \inters \cl{u}$. We have $\cl{u}=(\cl{a} \inters \cl{u})\join (\cl{b} \inters \cl{u})$ and $(\cl{a} \inters \cl{u})\inters (\cl{b} \inters \cl{u})= \void$. Hence, by (ii), we conclude that either $\cl{a} \inters \cl{u}=\void$ or $\cl{b} \inters \cl{u}=\void$ and so $\cl{u}$ is connected.
\end{proof}

\begin{defne}[\cite{Li96}]\label{defcomponent}
	A \dfn{component} of a sublocale $T\cont L$ is a non-void connected sublocale $D\cont T$ that is maximal among connected sublocales of $T$ in the following sense: given any connected sublocale $S\cont T$ such that $S\inters D\neq \void$, we have $S\cont D$.
\end{defne}

\begin{lemm}[\cite{Li96}]\label{lemmajoinopens}
    Let $L$ be a locale and $\fami{\o{u_i}}{i\in I}$ a family of open sublocales of $L$. Then $\bigjoin \fami{\o{u_i}}{i\in I}$ exists (and is equal to $\o{\bigjoin \fami{u_i}{i}}$). Moreover for every sublocale $S$
    $$S\inters \bigjoin \fami{\o{u_i}}{i\in I}=\bigjoin \fami{S\inters \o{u_i}}{i\in I}.$$
\end{lemm}

\begin{lemm}\label{lemmajoinint}
    Let $L$ be a locale and $\fami{S_i}{i\in I}$ a family of sublocales of $L$ such that $\bigjoin \fami{S_i}{i\in I}$ exists and is a sublocale of $L$. Then the following two properties hold.
    
    If $D$ is a complemented non-void sublocale of $\bigjoin \fami{S_i}{i\in I}$, then there exists some $i_0\in I$ such that $D\inters S_{i_0}\neq \void$.

    If $S_i$ is open for every $i\in I$, then for every sublocale $D$ such that $D\inters \bigjoin \fami{S_i}{i\in I}\neq \void$ there exists $i_0\in I$ such that $D\inters S_{i_0}\neq \void$.
\end{lemm}
\begin{proof}
    If $D$ is complemented in $\bigjoin \fami{S_i}{i\in I}$, we can consider the sublocale $T:=\bigjoin \fami{S_i}{i\in I}\difference D$. Then $T$ is strictly contained in $\bigjoin \fami{S_i}{i\in I}$ because $D$ is non-void. So by definition of join, it cannot happen that $T$ contains $S_i$ for every $i\in I$. Thus there exists $i_0\in I$ such that $S_{i_0}$ is not contained in $T$. Whence $D\inters S_{i_0}\neq \void$.

    If $S_i$ is open for every $i\in I$, then by \lemx\ref{lemmajoinopens}
    $$D\inters \bigjoin \fami{S_i}{i\in I} = \bigjoin \fami{D\inters S_i}{i\in I}.$$
    So if the left hand side is non-void, we can deduce that there exists $i_0\in I$ such that $D\inters S_{i_0}\neq \void$, otherwise the right hand side would be void, being a join of void sublocales.
\end{proof}

By Lemma \ref{lemmajoinint}, we are able to conclude the next result with respect to components when in a locally connected locale.

\begin{coroll}\label{corolljoinint}
    Let $L$ be locally connected and $\o{u}\cont L$. If $S\cont L$ is a sublocale such that $S\inters \o{u}\neq \void$ then there exists a component $C$ of $\o{u}$ such that $S\inters C\neq \void$.
\end{coroll}
\begin{proof}
Since $L$ is locally connected, $\o{u}=\bigjoin \set{C\cont \o{u}}{C \text{ component}}$. Every component of $\o{u}$ is open and so, by the second part of Lemma \ref{lemmajoinint}, $S$ must intersect at least one of the components of $\o{u}$.
\end{proof}

\begin{prop}\label{joinconnected}
	Let $\fami{S_i}{i\in I}$ be a family of connected sublocales of $L$ such that $\bigjoin \fami{S_i}{i\in I}$ exists and is a sublocale of $L$. If there exists $i_0\in I$ such that $\closure{S_i}\inters S_{i_0}\neq \void$ for every $i\in I$ (or $S_i\inters \closure{S_{i_0}}\neq \void$ for every $i\in I$), then $\bigjoin \fami{S_i}{i\in I}$ is connected.
\end{prop}

\begin{proof}
Let $\cl{u}$ and $\cl{v}$ be non-void closed sublocales of $L$ such that $\bigjoin\fami{S_i}{i\in I} \cont \cl{u} \join \cl{v}$ with $\bigjoin\fami{S_i}{i\in I} \inters \cl{u} \inters \cl{v} = \void$. Given $i\in I$, since the sublocale $S_i$ is connected, it must be either $S_i \inters \cl{u} = \void$ or $S_i \inters \cl{v} = \void$. This implies that for every $i\in I$ either $\closure{S_i} \cont \cl{u}$ or $\closure{S_i} \cont \cl{v}$. Without loss of generality we can suppose that, in particular, $S_{i_0} \cont \cl{u}$. Then, given any $i\in I$, since $\closure{S_i} \inters S_{i_0} \neq \void$, we have $S_i \cont \cl{u}$ and so $\bigjoin\fami{S_i}{i\in I} \cont \cl{u}$. Hence we conclude $\bigjoin\fami{S_i}{i\in I} \inters \cl{v}= \void$ and so $\bigjoin \fami{S_i}{i\in I}$ is connected.

We can use the same argument for the case $S_i\inters \closure{S_{i_0}}\neq \void$ for every $i\in I$.
\end{proof}

We now present the notion of a strongly locally connected locale.

\begin{defne}\label{defstronglloccon}
    A locale $L$ is \dfn{strongly locally connected} if for every $\o{u}\cont L$ and for every $x\in \o{u}$ there exists $\o{v}\cont \o{u}$ open connected sublocale with $x\in \o{v}$.
\end{defne}

\begin{prop}\label{strlocconisloccon}
    Every strongly locally connected locale is locally connected.

    Moreover, if $L$ is strongly locally connected and $\o{u}\cont L$, then every $x\in \o{u}$ belongs to a component of $\o{u}$.
\end{prop}
\begin{proof}
Let $L$ be a strongly locally connected locale and let $u\in L$. For every $x\in \o{u}$, take $\o{v^x}$ an open connected sublocale such that $x\in \o{v^x}$ and $\o{v^x} \cont  \o{u}$. Then we have $\o{u}=\bigjoin \set{\o{v^x}}{ x\in \o{u}}= \bigcup \set{\o{v^x}}{x\in \o{u}}$. Indeed, for every $x\in \o{u}$ we have $x\in \o{v^x}\cont \o{u}$, and every $\o{v^x}$ is contained in $\o{u}$. Then $u=\bigjoin \set{v^x}{x\in \o{u}}$. And every $v^x$ is a connected element, since $\o{v^x}$ is connected.
\end{proof}

\begin{rem}
    For topological spaces, strongly local connectedness is equivalent to local connectedness. But for locales one of the two properties is stronger than the other one.
\end{rem}

\begin{prop}\label{3-4}
	Let $L$ be a connected and locally connected locale. Let then $\o{v}$ be a component of an open sublocale $\o{u}\cont L$ such that $\o{v}\neq L$. Then $\bd{\o{v}}\neq \void$ and $\bd{\o{v}}\cont L\difference \o{u}=\cl{u}$.
\end{prop}
\begin{proof}
We notice that $\bd{\o{v}}= \cl{v\st \join v}$. Since $\cl{v\st \join v}=\cl{v} \inters \cl{v\st}$, we have 
$\cl{v \join v \st} \inters \o{u} \subseteq \cl{v} \inters \o{v} = \void$ 
and hence $\cl{v\st \join v} \subseteq \cl{u}$.
 Moreover, since $L$ is connected and $\o{v}$ is non-void and not equal to $L$, we have that $\o{v} \neq \cl{v\st}$ that implies $\cl{v \st \join v}= \cl{v\st} \inters \cl{v} \neq \void$.
\end{proof}

\begin{coroll}\label{3-5}
	Let $L$ be a connected and locally connected locale, and let $\cl{u}\cont L$ be a non-void closed sublocale. Every component $\o{v}$ of $L\difference \cl{u}=\o{u}$ is such that $\closure{\o{v}}\inters \cl{u}\neq \void$.
\end{coroll}

\begin{proof}
  We notice that $\bd{\o{v}} \cont \closure{\o{v}}$ and, by Proposition \ref{3-4}, $\bd{\o{v}}\neq \void$ and $\bd{\o{v}}\cont \cl{u}$.
\end{proof}

The following theorem will be very useful in the next sections.

\begin{teor}\label{3-10}
	Let $L$ be a connected and locally connected locale, and let $A=\cl{x}$ and $B=\cl{y}$ be non-void disjoint closed sublocales of $L$. Consider then $N\cont L$ a connected sublocale such that $N\inters A\neq \void \neq N\inters B$ (for example, $N=L$). Then there exists a component $D$ of $L\difference (A\join B)$ such that $N\inters D\neq \void$, $\closure{D}\inters A\neq \void$ and $\closure{D}\inters B\neq \void$.
\end{teor}
\begin{proof}
Consider the following sublocales:

\noindent$S_A= \left(\bigjoin \set{\o{u} \cont L \difference (A \join B)}{ \o{u} \text{ component}, \closure{\o{u}}\inters A \neq \void} \join A\right) \inters N$

\noindent$S_B= \left(\bigjoin \set{\o{u} \cont L \difference (A \join B)}{\o{u} \text{ component}, \closure{\o{u}}\inters B \neq \void}{} \join B\right) \inters N.$

By Corollary \ref{3-5}, every component $\o{u} \cont L\difference (A \join B)$ is such that $\closure{\o{u}}$ intersects either $A$ or $B$ and so $S_A \join S_B= N$. Since $N \cont \closure{S_A} \join \closure{S_B}$ is connected and it intersects both $\closure{S_A}$ and $\closure{S_B}$, we have $N \inters \closure{S_A} \inters \closure{S_B}\neq \void$. But $N \inters \closure{S_A} \inters \closure{S_B}=( \closure{S_A} \inters S_B) \join ( S_A \inters \closure{S_B})$ and hence either $\closure{S_A} \inters S_B\neq \void$ or $S_A \inters \closure{S_B}\neq \void$. Without loss of generality, we assume $S_A \inters \closure{S_B}\neq \void$. By Lemma \ref{lemmajoinint}, there exists $a\in S_A \inters \closure{S_B}$, $a\neq 1$, with either $a\in A \inters N$ or $a\in \o{u} \inters N$ with $\o{u} \cont L \difference (A \join B)$ component such that $\closure{\o{u}} \inters A \neq \void$.  

If $a\in \o{u} \inters N$ then $\closure{S_B} \inters \o{u} \neq \void$ and hence by Proposition \ref{closureintopen} $\o{u} \inters S_B \neq \void$. Since $\o{u} \inters B = \void$, by Lemma \ref{lemmajoinint} there exists $b\in \o{v} \inters N$ with $\o{v} \cont L \difference (A \join B)$ component such that $\closure{\o{v}} \inters B \neq \void$. Since $b\in \o{u} \inters \o{v}$, we have $\o{u} = \o{v}$ and so we have found a component of $L\difference (A \join B)$ that intersects $N$ and whose closure intersects both $A$ and $B$. 

If $a\in A\inters N$, then $a \in A \inters \closure{S_B}$ and so $A \inters \closure{S_B}$ is non void and complemented. Hence, by Lemma \ref{lemmajoinint}, there exists $a'\in A \inters \closure{S_B}$, $a' \neq 1$ such that $a'\in \o{w}$ with $\o{w} \cont L\difference B$ component.  This implies $\o{w} \inters S_B \neq \void$. But $\o{w} \inters S_B$ is complemented and so, thanks again to Lemma \ref{lemmajoinint}, there exists $b\in \o{v} \inters N$ with $\o{v} \cont L \difference (A \join B)$ component such that $\closure{\o{v}} \inters B \neq \void$. Since $\o{v} \inters \o{w} \neq \void$ and $\o{w}$ is connected, $\o{v}$ is not closed in $L \difference B$. But we have that 
$\closure[L\difference B]{\o{v}}= (\closure{\o{v}} \inters (L\difference(A \join B)) \join (\closure{\o{v}} \inters A)= \o{v} \join (\closure{\o{v}} \inters A)$
and so $\closure{\o{v}} \inters A\neq \void$. Hence $\o{v}$ is the desired component of $L\difference (A \join B)$.
\end{proof}

\begin{prop} \label{componentbd}
    Let $\o{u}\cont L$ be an open sublocale and let $C$ be a component of $\o{u}$. Then $\bd{C}\cont \bd{\o{u}}$. 
\end{prop}
\begin{proof}
Since $C \cont \o{u}$, we have $\bd{C} \cont \closure{\o{u}}$. But $\bd{C} \cont L\difference \o{u}$ by Proposition \ref{3-4} and so $\bd{C}\cont \bd{\o{u}}$.
\end{proof}

\section{Unicoherent Locales}
In this section, we extend the notion of unicoherence to the setting of locales. We prove that the locale of opens of a topological space is unicoherent precisely when the topological space is, immediately obtaining many examples of this theory. We also introduce \emph{open unicoherence}, which for topological spaces is an equivalent characterization of unicoherence of notable relevance. We investigate some results on unicoherent locales and prove some lemmas that will be useful in the following section.

\begin{defne} \label{defcontinuum}
	A \dfn{continuum} in a locale $L$ is a non-void closed connected sublocale $\cl{u}\cont L$.

    A \dfn{region} in a locale $L$ is a non-void open connected sublocale $\o{u}\cont L$.
\end{defne}

\begin{prop} \label{continuatoploc}
	Let $X$ be a topological space. The closed sublocale $\cl{X\difference F}$ of $\mathcal{O}(X)$ associated to a closed subspace $F$ of $X$ is a continuum precisely when $F$ is a continuum (in the classical topological sense).
\end{prop}
\begin{proof}
The closed sublocale $\cl{X\difference F}$ is non-void precisely when $F\neq\emptyset$. Moreover, it is connected as a sublocale if and only if whenever $F=G \cup H$ with $G,H$ closed disjoint subsets of $X$ then either $G=F$ or $H=F$, that is precisely when $F$ is connected as a subspace of $X$.  
\end{proof}

We now extend the notion of unicoherence from topological spaces to locales.

\begin{defne}\label{defunicoh}
	A locale $L$ is \dfn{unicoherent} if whenever $L=H\join K$ with $H,K\cont L$ continua, we have that $H\inters K$ is a continuum.

    $L$ is \dfn{open unicoherent} if whenever $L=A\join B$ with $A,B\cont L$ regions, we have that $A\inters B$ is a region.
\end{defne}

\begin{prop}\label{propunicohtoploc}
	Let $X$ be a topological space. The locale $\mathcal{O}X$ is unicoherent precisely when $X$ is unicoherent.
\end{prop}

\begin{proof}
Straightforward using the definition of unicoherent locale and Proposition \ref{continuatoploc}.
\end{proof}

\begin{exampl}
\prox\ref{propunicohtoploc} gives many examples of unicoherent locales. For instance, the locale of reals (which coincides with the locale of opens of the euclidean space $\mathbb{R}$ under mild conditions) is unicoherent, thanks to \prox\ref{propunicohtoploc}. While the locale of opens of a circle is not unicoherent, again thanks to \prox\ref{propunicohtoploc}.
\end{exampl}

In the following, we aim at establishing several equivalent characterizations of the notion of unicoherence for locales, generalizing the ones existing for topological spaces. In particular, some of the equivalent characterizations that we achieve will involve separation conditions for locales.

\begin{defne}\label{defdiff}
	Let $S\cont L$ be a complemented sublocale. $S$ \dfn{separates $L$} if $L\difference S$ is not connected.

	Given $x,y\in L$, $S$ \dfn{separates $x$ and} (or \dfn{from}) \dfn{$y$ in $L$} if $L\difference S$ is not connected and admits a separation $L\difference S\cont \cl{a}\join \cl{b}$ such that $x\in (L\difference S)\inters\cl{a}$ and $y\in (L\difference S)\inters\cl{b}$.

    Given $X,Y\cont L$, $S$ \dfn{separates $X$ and $Y$ in $L$} if $S$ separates every element $x\in X$ from every element $y\in Y$.
\end{defne}

\begin{prop}\label{3-3}
	Let $L$ be a locally connected locale and $\o{u}\cont L$. Given any component $\o{v}$ of $\o{u}$ such that $\closure{\o{v}}\neq L$, we have that $L\difference \bd{\o{v}}=\o{v\st \join v}$ is not connected and
	$$L\difference \bd{\o{v}}=\o{v\st \join v}=\o{v\st}\join \o{v}=(L\difference \closure{\o{v}})\join \o{v}$$
	exhibits a separation. As a consequence, also
		$$L\difference \bd{\o{v}}\cont \closure{L\difference \closure{\o{v}}}\join \closure{\o{v}}$$
		exhibits a separation.
\end{prop}
\begin{proof}
We want to show that $\o{v\st}\join \o{v}$ exhibits a separation of $\o{v\st \join v}$. Since $\o{v\join v\st}= \o{v\st} \join \o{v}$, we have $\o{v\st \join v} \inters \o{v\st} \inters \o{v} = \o{v\st \meet v} = \o {0}= \void.$
Moreover, $\o{v\st \join v} \inters \o{v}=\o{v}$ is non void since $v\neq 0$ and analogously $\o{v\st \join v} \inters \o{v\st}$ is non void since $v\st\neq 0$ because $\cl{v\st}=\closure{\o{v}}\neq L$.

We now observe that, by Proposition \ref{closedofsub}, $\o{v}=(L\difference \bd{\o{v}})\difference (L\difference \closure{\o{v}})$ is closed in $L\difference \bd{\o{v}}$ and analogously $(L\difference \closure{\o{v}})$ is closed in $S$ and then thanks to \prox\ref{closedofsub}
		$$L\difference \bd{\o{v}}\inters \closure{\o{v}}\inters \closure{L\difference \closure{\o{v}}}=\closure[L\difference \bd{\o{v}}]{\o{v}}\inters \closure[L\difference \bd{\o{v}}]{L\difference \closure{\o{v}}}=\o{v}\inters (L\difference \closure{\o{v}})=\void.$$
\end{proof}


\begin{defne}\label{defsimple}
    A sublocale $C\cont L$ is \dfn{simple} if $C$ is complemented and both $C$ and $L\difference C$ are connected.
\end{defne}

\begin{prop}\label{compdiffissimple}
    Let $L$ be a connected locale and $C\cont L$ a continuum in $L$. Then every component of $L\difference C$ is simple.
\end{prop}
\begin{proof}
Let $U$ be a component of $L\difference C$. We need to prove that $L \difference U$ is connected. Let $P$ and $Q$ be separated sublocales such that $L\difference U=P \join Q$. By Lemma \ref{lemmadiffsep}, $U \join P$ and $U\join Q$ are connected and since $C \cont L \difference U$ is connected we have that either $C \inters P=\void$ or $C\inters Q=\void$. If $C \inters P=\void$, we have that $U\join P$ is a connected contained in $L\difference C$. Since $U\cont U \join P$ is a component of $L\difference C$, we must have $U=U \join P$ and hence $P=\void$. Analogously, if $C \inters Q=\void$, we conclude that $Q=\void$. 
\end{proof}

\begin{defne}\label{defnormsep}
    Let $S,T\cont L$ be sublocales of $L$. $S$ and $T$ are \dfn{normally separated} if there exist $\o{u},\o{v}\cont L$ such that $\o{u}\inters\o{v}=\void$, $S\cont \o{u}$ and $T\cont \o{v}$.

    A sublocale $C\cont L$ is \dfn{normally connected} if whenever $C=S\join T$ with $S$ and $T$ normally separated, then $S=\void$ or $T=\void$.
\end{defne}

\begin{prop}\label{propconnormcon}
    Let $A\cont L$ be a sublocale of $L$. Consider the following properties:
    \begin{enumT}
        \item $A$ is connected;
	\item Given any open sublocale $\o{u}$ such that $A\cont \o{u}$, there exists a connected sublocale $C\cont L$ such that $A\cont C\cont \o{u}$;
        \item $A$ is normally connected;
        \item Given any complemented sublocale $S\cont L$, if $A\inters S\neq \void$ and $A\inters L\difference S\neq \void$ then also $A\inters \bd{S}\neq \void$.
    \end{enumT}
    Then $(i)\implic (ii)\implic (iii) \implic (iv)$. In particular, every connected sublocale $A$ satisfies $(iv)$.
\end{prop}

\begin{proof}
(i) $\implic$ (ii). Trivial, it suffices to consider $A$ itself as connected sublocale containing $A$.

(ii) $\implic$ (iii). Let $A=S \join T$ with $S$ and $T$ normally separated and let $V$ and $W$ be disjoint open sublocales such that $S \cont V$ and $T \cont W$. Then $A\cont V \join W$ and so by (ii) there exists $C$ connected such that $A \cont C \cont V \join W$. This implies that either $B \inters V=\void$ or $B \inters W=\void$ and hence either $S=\void$ or $T=\void$.

 (iii) $\implic$ (iv). Let $S$ be a complemented sublocale of $L$ such that $A\inters S\neq \void$ and $A\inters L\difference S\neq \void$. Since $A\inters S\neq \void$ we have $\interior{A} \inters S\neq \void$. On the other hand, since $A\inters L\difference S\neq \void$, we have $(L\difference \closure{S}) \inters A\neq \void$. Hence, since $A$ is normally connected and $\interior{S}$ and $L\difference \closure{S}$ are normally separated, $A$ is not contained in $L\difference \bd{S}=\interior{S} \join (L\difference \closure{S})$ and so $A \inters \bd{S} \neq \void$.
\end{proof}

\begin{lemm}\label{lemmasimplebd}
    Let $L$ be an open unicoherent locale. Then every simple sublocale has a normally connected boundary.

    Moreover, if $A,B\cont L$ are simple sublocales of $L$ such that $\bd{A}\inters \bd{B}\neq \void$, then either $A\cont B$, $B\cont A$, $A\inters B=\void$ or $A\join B=L$.
\end{lemm}
\begin{proof}
Let $A\cont L$ be a simple sublocale and let $U\cont L$ be an open sublocale such that $\bd{A} \cont U$. Consider the following sublocale:

$$B=\bigjoin \set{D\cont U}{ D \text{ component}, D \inters \bd{A} \neq \void}$$
Then $B$ is open and $\bd{A} \cont B \cont U$. Moreover, $B$ is connected. Indeed, notice that $B=(A \join B) \inters ((L\difference A) \join B)$. Both $A \join B$ and $(L\difference A) \join B$ are connected because for every component $D\cont U$ such that $D \inters \bd{A} \neq \void$ we have $D \inters \closure{A}\neq \void$ and $D\inters (L\difference A) \neq \void$. In addition to this, $A \join B$ and $(L\difference A) \join B$ are open sublocales. The fact that $L$ is open unicoherent implies then that $B$ is connected. And so, we have shown that condition (ii) of Proposition \ref{propconnormcon} holds for $A$. Hence $A$ is normally connected.

We now prove the second part of the statement. Let $A$ and $B$ be simple sublocales of $L$ such that $\bd{A} \inters \bd{B} = \void$. If $B= \void$ then $B\cont A$ and if $B=L$ then $A \cont B$.  If $B$ is not void and not the entire $L$, then $B$ cannot be clopen because $L$ is connected. This implies that $\bd{B}$ is not void. But then $\bd{B}$ cannot be contained in $A \inters (L\difference A)$. If $\bd{B}$ is not contained in $A$ then $\bd{B} \inters (L\difference A) \neq \void$. Since $\bd{B}$ is normally connected (by the first part of this Lemma) and $\bd{A}\inters \bd{B} =\void$, by (iv) of Proposition \ref{propconnormcon}, it must be $\bd{B}\inters A=\void$ and since $A$ is connected this implies either $A \inters B=\void$ or $A \inters (L \difference B)=\void$ and thus $A \cont B$. Analogously, if $\bd{B}$ is not contained in $L \difference A$, since $L\difference A$ is connected we can apply the same argument to $L\difference A$ instead of $A$ and conclude that either $A \join B=L$ or $B \cont A$.
 \end{proof}

\begin{lemm}\label{bdinfinitejoin}
    Let $L$ be strongly locally connected and $\fami{\o{u_i}}{i\in I}$ be a family of open sublocales. Then 
    $$\bd{\bigjoin{\fami{\o{u_i}}{i\in I}}} \cont \closure{\bigjoin{\fami{\bd{\o{u_i}}}{i\in I}}}.$$
\end{lemm}
\begin{proof}
We prove that $\bd{\bigjoin{\fami{\o{u_i}}{i\in I}}} \inters (L \difference (\closure{\bigjoin{\fami{\bd{\o{u_i}}}{i\in I}}}))=\void$, because than we obtain $\bd{\bigjoin{\fami{\o{u_i}}{i\in I}}} \cont \closure{\bigjoin{\fami{\bd{\o{u_i}}}{i\in I}}}.$ If $\bd{\bigjoin{\fami{\o{u_i}}{i\in I}}} \inters (L \difference (\closure{\bigjoin{\fami{\bd{\o{u_i}}}{i\in I}}}))\neq\void$, since $L \difference (\closure{\bigjoin{\fami{\bd{\o{u_i}}}{i\in I}}})$ is open, by Lemma \ref{neighinterior},  there exist $x\in \bd{\bigjoin{\fami{\o{u_i}}{i\in I}}} \inters (L \difference (\closure{\bigjoin{\fami{\bd{\o{u_i}}}{i\in I}}}))$ and an open sublocale $U_x \cont L \difference (\closure{\bigjoin{\fami{\bd{\o{u_i}}}{i\in I}}})$ such that $x\in U_x$. Since $L$ is strongly locally connected, there exists an open connected sublocale $\o{v^x}$ such that $\o{v^x} \cont U_x$ and $x\in\o{v^x}$. Since $x\in \bd{\bigjoin{\fami{\o{u_i}}{i\in I}}}$, by Lemma \ref{neighboundary}, it must be $\o{v^x} \inters \bigjoin{\fami{\o{u_i}}{i\in I}}\neq \void$. Hence, by Lemma \ref{lemmajoinint}, there exists $i_0\in I$ such that $\o{v^x}\inters \o{u_{i_0}} \neq \void$. On the other hand,  by Lemma \ref{lemmajoinopens}, we have $L\difference \bigjoin{\fami{\o{u_i}}{i\in I}}= \inters \fami{L \difference\o{u_i}}{i\in I}$. Whence, since $\o{v^x} \inters L\difference \bigjoin{\fami{\o{u_i}}{i\in I}}\neq \void$ by Lemma \ref{neighboundary}, we have that for every $i\in I$ $\o{v^x} \inters (L\difference \o{u_{i}})\neq \void$. In particular, $\o{v^x} \inters (L\difference \o{u_{i_0}})\neq \void$. Hence  the connected $\o{v^x}$ intersects both $\o{u_{i_0}}$ and $L\difference \o{u_{i_0}}$ and so, by Proposition \ref{propconnormcon}, $\o{v^x}\inters \bd{\o{u_{i_0}}}\neq \void$.  But this is a contradiction since $\bd{\o{u_{i_0}}} \cont \bigjoin{\fami{\bd{\o{u_i}}}{i\in I}} \cont \closure{\bigjoin{\fami{\bd{\o{u_i}}}{i\in I}}}$ and $\o{v^x} \cont L \difference (\closure{\bigjoin{\fami{\bd{\o{u_i}}}{i\in I}}})$. Hence we conclude $\bd{\bigjoin{\fami{\o{u_i}}{i\in I}}} \inters (L \difference (\closure{\bigjoin{\fami{\bd{\o{u_i}}}{i\in I}}}))=\void$. 
   \end{proof}

\section{Main Results}
\label{Section: Main Results}
We now present several characterizations of unicoherence for a connected, locally connected locale. The results presented here are obtained through generalizations of properties that for topological spaces were known to be equivalent to unicoherence. Many of such properties are recorded in the survey paper \cite{GarIll89}. The properties for locales that we obtain here, together with the detailed proofs of their equivalence, will form the foundation for further study of unicoherent locales in future.

\begin{teor}\label{mainthm}
    Let $L$ be a connected and locally connected locale. The following properties are equivalent:
    \begin{enumerate}[itemsep=2.5mm]
\item[\Rom{1}] Whenever $X,Y\cont L$ with $X\neq \void\neq Y$ and $\cl{u},\cl{v}\cont L$ are such that $\cl{u}\inters\cl{v}=\void$ and neither $\cl{u}$ nor $\cl{v}$ separate any element of $X$ (different from 1) from any element of $Y$ (different from 1) in $L$, we have that $\cl{u}\join \cl{v}$ does not separate $X$ and $Y$ in $L$;
\item[\Rom{2}](Brouwer Property) If $\cl{u}\cont L$ is a continuum, then every component $D$ of $L\difference \cl{u}=\o{u}$ has as boundary $\bd{D}$ a continuum;
\item[\Rom{3}](Unicoherence) $L$ is unicoherent;
\item[\Rom{4}] If $\cl{u}\cont L$ is non-void and $D_1,D_2$ are disjoint components of $L\difference\cl{u}=\o{u}$ such that $\bd{D_1}=\bd{D_2}$ then $\bd{D_1}$ is a continuum;
\item[\Rom{5}] If $C$ and $D$ are disjoint complemented connected sublocales of $L$ such that $\bd{C}\cont \bd{D}$, then $\bd{C}$ is connected;
\item[\Rom{6}] If $C\cont L$ is simple, then $\bd{C}$ is connected;
\item[\Rom{7}] If $R$ is a simple region, then $\bd{R}$ is connected;
\item[\Rom{8}] If $A$ and $B$ are disjoint regions such that $\bd{A}=\bd{B}$, then $\bd{A}$ is connected;
\item[\Rom{9}] If $A$ and $B$ are regions such that $\bd{A}\inters \bd{B}=\void$, then $A\inters B$ is connected;
\item[\Rom{10}] (Open unicoherence) $L$ is open unicoherent.
    \end{enumerate}
\end{teor}
\begin{proof}


\Rom{1} $\implic$ \Rom{2}. We prove that if \Rom{2} does not hold then \Rom{1} does not hold. So consider a continuum $\cl{u}\cont L$ and a component $D$ of $L\difference \cl{u}=\o{u}$ such that $\bd{D}$ is not a continuum. Since $L$ is locally connected, $D$ is an open sublocale $\o{v}$ of $L$. Notice that $\cl{u}\neq L$ because $L\difference L=\void$ has no components. And $\o{v}\neq L$, because by uniqueness of the pseudo-complement $L\difference \cl{u}$ cannot be the whole $L$, as $\cl{u}\neq\void$. Then by \prox\ref{3-4} $\bd{\o{v}}\neq \void$ and $\bd{\o{v}}\cont \cl{u}$. And of course $\bd{\o{v}}$ is a closed sublocale of $L$, by definition. Since $\bd{\o{v}}$ is not a continuum, we obtain that $\bd{\o{v}}$ is not connected. By \lemx\ref{closedsubconnected}, since $\bd{\o{v}}$ is closed, there exist non-void $\cl{a},\cl{b}\cont L$ such that $\bd{\o{v}}=\cl{a}\join \cl{b}$ and $\cl{a}\inters\cl{b}=\void$.
		
		\noindent We want to construct $X,Y\cont L$ with $X\neq\void\neq Y$ such that neither $\cl{a}$ nor $\cl{b}$ separate any element of $X$ from any element of $Y$ in $L$, but $\bd{\o{v}}=\cl{a}\join \cl{b}$ separates $X$ and $Y$ in $L$. We first show that $\closure{\o{v}}\neq L$. For this, we show that $\cl{u}$ cannot be entirely contained in $\closure{\o{v}}$. Notice that $\closure{\o{v}}=\o{v}\join \bd{\o{v}}$ by definition of boundary. So if we had that $\cl{u}\cont \closure{\o{v}}$, then $\cl{u}\cont \bd{\o{v}}=\cl{a}\join \cl{b}$, because $\cl{u}\inters \o{v}=\void$. But $\void\neq\cl{a}\cont \bd{\o{v}}\cont \cl{u}$ and $\void\neq\cl{b}\cont \bd{\o{v}}\cont \cl{u}$. So this would contradict the connectedness of $\cl{u}$. We then have that $\closure{\o{v}}\neq L$. By \prox\ref{3-3},
		$$L\difference \bd{\o{v}}=\o{v\st\join v}=\o{v\st}\join \o{v}=(L\difference \closure{\o{v}})\join \o{v}$$
		exhibits a separation. And as a consequence also
		$$L\difference \bd{\o{v}}\cont \closure{L\difference \closure{\o{v}}}\join \closure{\o{v}}$$
		exhibits a separation.
        
        
        We define $X:=\o{v}$, which is non-void. By \thex\ref{3-10}, with $A:=\cl{a}$ and $B:=\cl{b}$ (which are disjoint closed sublocales) and $N:=\cl{u}$, there exists a component $C$ of $L\difference (\cl{a}\join \cl{b})=L\difference \bd{\o{v}}$ such that $\cl{u}\inters C\neq \void$, $\closure{C}\inters \cl{a}\neq \void$ and $\closure{C}\inters \cl{b}\neq \void$. We define $Y:=\cl{u}\inters C$.
		
		\noindent $\bd{\o{v}}$ separates $X$ and $Y$ in $L$. Indeed $Y\cont L\difference \closure{\o{v}}$, since $Y\cont C\cont L\difference \bd{\o{v}}$ and $Y\inters \o{v}\cont \cl{u}\inters \o{v}=\void$. But $\cl{a}$ does not separate any element $x$ of $X$ (different from 1) from any element $y$ of $Y$ (different from 1). Indeed, consider (if it exists) a separation $L\difference \cl{a}\cont\cl{l}\join\cl{m}$ with $(L\difference \cl{a})\inters \cl{l}\inters \cl{m}=\void$, $(L\difference \cl{a})\inters \cl{l}\neq \void$ and $(L\difference \cl{a})\inters \cl{m}\neq \void$. And assume that we have $x\in \cl{l}$ and $y\in \cl{m}$. Notice that $\o{v}\cont L\difference \cl{a}$ because $\o{v}\inters \cl{a}\cont \o{v}\inters \bd{\o{v}}=\void$. Since $\o{v}$ is connected, $\o{v}\inters \cl{l}=\void$ or $\o{v}\inters \cl{m}=\void$. Since $x\in \o{v}\inters \cl{l}$, we get $\o{v}\inters \cl{m}=\void$ and $\o{v}\cont \cl{l}$. Whence also $\closure{\o{v}}\cont \cl{l}$. Analogously, $C\cont L\difference \cl{a}$, $C$ is connected and $y\in C\inters \cl{m}$, so $C\cont \cl{m}$ and also $\closure{C}\cont \cl{m}$. But then $\closure{C}\inters \cl{b}\cont (L\difference \cl{a})\inters \cl{l}\inters \cl{m}$. Indeed $\closure{C}\inters \cl{b}\cont \cl{b}\cont L\difference \cl{a}$ since $\cl{a}\inters \cl{b}=\void$, $\closure{C}\inters \cl{b}\cont\closure{C}\cont \cl{m}$ and $\closure{C}\inters \cl{b}\cont\cl{b}\cont \bd{\o{v}}\cont \closure{\o{v}}\cont \cl{l}$. Since $\closure{C}\inters \cl{b}\neq \void$, this contradicts the separation $L\difference \cl{a}\cont\cl{l}\join\cl{m}$. We thus conclude that $\cl{a}$ does not separate any element $x\in X$ from any element $y\in Y$. Analogously, $\cl{b}$ does not separate any element $x\in X$ from any element $y\in Y$. So \Rom{1} does not hold.

\Rom{2} $\implic$ \Rom{3}. We prove that if \Rom{3} does not hold then \Rom{2} does not hold. Let $\cl{u}, \cl{v} \cont L$ be continua such that $L=\cl{u} \join \cl{v}$ and $\cl{u\join v}$ is not a continuum. Since $L$ is connected and $L=\cl{u} \join \cl{v}$, we have $\cl{u\join v}=\cl{u} \inters \cl{v}\neq \void$ and so it must be $\cl{u\join v}$ not connected. By Lemma \ref{closedsubconnected}, there exist $\cl{a}$ and $\cl{b}$ non void closed sublocales of $L$ such that $\cl{u\join v}=\cl{a} \join \cl{b}$, $\cl{a} \inters \cl{b}=\void$. Since $\cl{v}$ is connected and $\cl{v}$ intersects both $\cl{a}$ and $\cl{b}$, by Theorem \ref{3-10} there exists a component $\o{x}$ of $L\difference \cl{u \join v}$ such that $\closure{\o{x}} \inters \cl{a}\neq \void$, $\closure{\o{x}} \inters \cl{b}\neq \void$ and $\o{x} \inters \cl{v}\neq \void$. Since $\cl{u} \inters \cl{v} \neq \void$, by Proposition \ref{3-4}, we have $\bd{\o{x}}\neq \void$ and $\bd{\o{x}}\cont \cl{u} \join \cl{v}$. We now prove that $\o{x}$ is a component of $L \difference \cl{u}$ and that $\bd{\o{x}}$ is not connected.  Since $\o{x} \inters \cl{u}=\void$, we have that $\o{x} \cont L\difference \cl{u}$ and, by Proposition \ref{connectedsub}, $\o{x}$ is connected in $L\difference \cl{u}$. Let now $D\cont L \difference \cl{u}$ be a connected sublocale such that $\o{x} \inters D\neq \void$. By Proposition \ref{connectedsub}, $D$ is connected also as sublocale of $L\difference (\cl{u} \inters \cl{v})$ and hence $D \cont \o{x}$. So $\o{x}$ is a component of $L \difference \cl{u}$. Since $\o{x}$ is connected, $\o{x} \inters \cl{v} \neq \void$ and $\o{x}=(\o{x} \inters \cl{u}) \join (\o{x} \inters \cl{v})$ with $\o{x} \inters \cl{u}$ and $\o{x} \inters \cl{v}$ disjoint and closed in $\o{x}$ by Proposition \ref{closedofsub}, it must be $\o{x} \inters \cl{u}= \void$. But then $\o{x} \cont \cl{v}$ and so $\closure{\o{x}} \cont \cl{v}$. Since $\closure{\o{x}} \inters \cl{a} \neq \void$ and $\o{x} \inters \cl{a}= \void$, it must be $\bd{x}\inters \cl{a} \neq \void$. But $\bd{x}\inters \cl{a} \neq \void= \bd{\o{x}} \inters \cl{u} \inters \cl{a}$ because $\bd{\o{x}} \cont \cl{u} \join \cl{v}$ and so $\closure{\o{x}}\inters \cl{u} \inters \cl{a}\neq \void$. Analogously, $\closure{\o{x}}\inters \cl{u} \inters \cl{b}\neq \void$. Since $\closure{\o{x}} \inters \cl{u} \cont \cl{a} \join \cl{b}$ and $\cl{a} \inters \cl{b} \neq \void$, we conclude that $\closure{\o{x}}\inters \cl{u}$ is not connected. But $\closure{\o{x}}\inters \cl{u}=\bd{\o{x}}$ because $\o{x} \inters \cl{u}=\void$ and $\bd{\o{x}} \cont \cl{u}$ and so we conclude that $\bd{\o{u}}$ is not a continuum.

\Rom{3} $\implic$ \Rom{4}. We prove that if \Rom{4} does not hold then \Rom{3} does not hold. Let $\cl{u} \cont L$ non-void and let $\o{x}$ and $\o{y}$ be components of $\o{u}$ such that $\bd{\o{x}}=\bd{\o{y}}$ and $\bd{\o{x}}$ is not a continuum. Since $\bd{\o{x}}$ is closed and $\bd{\o{x}}\neq \void$ because otherwise $\o{x}$ would be a non-void clopen of the connected locale $L$ that is not the whole $L$, it must be $\bd{\o{x}}$ not connected. By Lemma \ref{closedsubconnected}, there exist $\cl{a}$ and $\cl{b}$ non-void closed sublocales such that $\bd{\o{x}}= \cl{a} \join \cl{b}$, $\cl{a} \inters \cl{b}=\void$. Since $\o{x} \neq L$, by Proposition \ref{3-4} we have $\bd{\o{x}}\neq \void$ and $\bd{\o{x}} \cont \cl{u}$. Moreover $\closure{\o{x}} \inters \o{y}= (\o{x} \inters \o{y}) \join (\bd{\o{y}} \inters \o{y})=\void$ and so $\closure{\o{x}} \neq L$. Thanks to Proposition \ref{3-3}, we then have that $L \difference \bd{\o{x}}= \o{x} \join \o{x\st}$ gives a separation of $L \difference \bd{\o{x}}$. Moreover, $\o{x}$ is a component of $L \difference \bd{\o{x}}$. Indeed, $\o{x}$ is connected as a sublocale of $L \difference \bd{\o{x}}$ because $\o{u} \cont L \difference \bd{\o{x}}$. And if $D\cont L \difference \bd{\o{x}}$ is connected and $D\inters \o{x}\neq \void$ it must be $D \inters \o{x\st}=\void$ that implies $D\cont \o{x}$. We now observe that \\
$L\difference \o{x}= \bigjoin \set{D\cont L\difference \bd{\o{x}}}{D \text{ component}, D\neq \o{x}, D \neq{\o{y}}} \join \closure{\o{y}}$. \\ \noindent Moreover, by Corollary \ref{3-5}, every component $E$ of $L \difference \bd{\o{x}}$ is such that $\closure{E} \inters \bd{\o{x}}\neq \void$ and hence $\bd{E} \inters \bd{\o{x}}\neq \void$. Since $\bd{\o{x}}=\bd{\o{y}}$ , by Proposition \ref{joinconnected}, $L\difference \o{x}$ is connected. Then $L\difference \o{x}$ and $\closure{\o{x}}$ are continua such that their join is $L$ but $(L\difference \o{x})\inters \closure{\o{x}}=\bd{\o{x}}$ is not connected.

\Rom{4} $\implic$ \Rom{1}. We prove that if \Rom{1} does not hold then \Rom{4} does not hold. So let $X,Y\cont L$ with $X\neq\void\neq Y$ and $\cl{u},\cl{v}\cont L$ be sublocales such that $\cl{u}\inters\cl{v}=\void$ and neither $\cl{u}$ nor $\cl{v}$ separate any element of $X$ (different from 1) from any element of $Y$ (different from 1), but $\cl{u}\join \cl{v}$ separates $X$ and $Y$ in $L$. In particular we have that $X,Y\cont L\difference (\cl{u}\join \cl{v})$. By \corx\ref{corolljoinint}, there exists a component $\o{w}$ of $L\difference (\cl{u}\join \cl{v})$ such that $\o{w}\inters X\neq \void$. Consider then $x\in \o{w}\inters X$ with $x\neq 1$, and $y\in Y$ with $y\neq 1$. Since $\cl{u}\join \cl{v}$ separates $x$ and $y$ in $L$, there exists a separation
		$$L\difference (\cl{u}\join \cl{v})\cont \cl{a}\join \cl{b}$$
		with $(L\difference (\cl{u}\join \cl{v}))\inters \cl{a}\inters\cl{b}=\void$, $(L\difference (\cl{u}\join \cl{v}))\inters \cl{a}\neq \void$ and $(L\difference (\cl{u}\join \cl{v}))\inters \cl{b}\neq \void$, such that $x\in (L\difference (\cl{u}\join \cl{v}))\inters\cl{a}$ and $y\in(L\difference (\cl{u}\join \cl{v}))\inters \cl{b}$. Since $\o{w}$ is connected, either $\o{w}\inters \cl{a}=\void$ or $\o{w}\inters \cl{b}=\void$. As $x\in \o{w}\inters \cl{a}$, we get that $\o{w}\inters \cl{b}=\void$ and $\o{w}\cont \cl{a}$. Notice then that $\closure{\o{w}}\cont \cl{a}\neq L$. By \prox\ref{3-4}, $\bd{\o{w}}\neq \void$ and $\bd{\o{w}}\cont \cl{u}\join \cl{v}$. By \prox\ref{3-3},
		$$L\difference \bd{\o{w}}=\o{w}\join (L\difference \closure{\o{w}})$$
		exhibits a separation, and as a consequence also
		$$L\difference \bd{\o{w}}\cont \closure{\o{w}}\join \closure{L\difference \closure{\o{w}}}$$
		exhibits a separation. Notice that $\bd{\o{w}}$ then separates $X\inters \o{w}$ and $Y\inters \cl{b}$. Indeed
		$$Y\inters \cl{b}\cont (L\difference (\cl{u}\join \cl{v}))\inters \cl{b}\cont L\difference \cl{a}\cont L\difference \closure{\o{w}}.$$
		Consider now $F_1:=\bd{\o{w}}\inters \cl{u}$ and $F_2:=\bd{\o{w}}\inters \cl{v}$. Then $F_1\join F_2=\bd{\o{w}}$. Moreover $F_1\neq \void$. Indeed if we had $\bd{\o{w}}\cont \cl{v}$ we would get
		$$L\difference \cl{v}\cont L\difference \bd{\o{w}}\cont \closure{\o{w}}\join \closure{L\difference \closure{\o{w}}}.$$
		$\o{w}\cont L\difference (\cl{u}\join \cl{v})\cont L\difference \cl{v}$, so $X\inters \o{w}\cont \o{w}\cont (L\difference \cl{v})\inters \closure{\o{w}}$. Moreover $Y\inters \cl{b}\cont Y\cont L\difference (\cl{u}\join \cl{v}) \cont L\difference \cl{v}$, so $Y\inters \cl{b}\cont (L\difference \cl{v})\inters \closure{L\difference \closure{\o{w}}}$. But this would contradict that $\cl{v}$ does not separate any element of $X$ (different from 1) from any element of $Y$ (different from 1). Thus $F_1\neq \void$, and analogously $F_2\neq \void$. Whence $\bd{\o{w}}=F_1\join F_2$ exhibits a separation. Notice then that $F_1$ does not separate any element of $X\inters \o{w}$ (different from 1) from any element of $Y\inters \cl{b}$ (different from 1). If it were otherwise, we would get that $\cl{u}$ separates an element of $X$ (different from 1) from an element of $Y$ (different from 1). Indeed $F_1\cont \cl{u}$, whence $L\difference \cl{u}\cont L\difference F_1$, and $X,Y\cont L\difference (\cl{u}\join \cl{v})\cont L\difference \cl{u}$. Analogously, $F_2$ does not separate any element of $X\inters \o{w}$ (different from 1) from any element of $Y\inters \cl{b}$ (different from 1).
		
		\noindent Let now $\o{z}$ be a component of $L\difference \bd{\o{w}}$ such that $\o{z}\inters Y\inters \cl{b}\neq \void$, which exists by \corx\ref{corolljoinint}. Since $\o{z}$ is connected and $\o{z}\inters (L\difference\closure{\o{w}})\neq\void$, we have $\o{z}\inters \o{w}=\void$. So $\o{z}\cont L\difference \o{w}$ and $\closure{\o{z}}\cont L\difference \o{w}\neq L$. By \prox\ref{3-4}, $\bd{\o{z}}\neq \void$ and $\bd{\o{z}}\cont \bd{\o{w}}$. Moreover by \prox\ref{3-3}
		$$L\difference \bd{\o{z}}=\o{z}\join (L\difference \closure{\o{z}})$$
		exhibits a separation. We can apply the argument above to $\o{z}$ in place of $\o{w}$, using $X\inters \o{w}$ in place of $X$, $Y\inters \cl{b}$ in place of $Y$, $F_1$ in place of $\cl{u}$ and $F_2$ in place of $\cl{v}$. Defining $B_1:=\bd{\o{z}}\inters F_1$ and $B_2:=\bd{\o{z}}\inters F_2$, we obtain a separation $\bd{\o{z}}=B_1\join B_2$, with $B_1\neq\void\neq B_2$. Notice that $\o{w}\cont L\difference \bd{\o{z}}$, because $\o{w}\inters \bd{\o{z}}\cont \o{w}\inters \bd{\o{w}}=\void$. Thanks to \corx\ref{corolljoinint}, there exists a component $\o{w'}$ of $L\difference \bd{\o{z}}$ such that $\o{w'}\inters \o{w}\neq \void$. Since $\o{w}$ is connected, we then have that $\o{w}\cont \o{w'}$. As $\o{w'}$ is connected and $\o{w'}\inters (L\difference \closure{\o{z}})\contain \o{w}\inters (L\difference \closure{\o{z}})\neq \void$, because $\o{w}\cont L\difference \bd{\o{z}}$ and $\o{w}\inters \o{z}=\void$, we get that $\o{w'}\inters \o{z}=\void$. By \prox\ref{3-4}, $\bd{\o{w'}}\cont \bd{\o{z}}$. Moreover $\bd{\o{z}}\cont \bd{\o{w}}\cont \closure{\o{w}}\cont \closure{\o{w'}}$ and $\bd{\o{z}}\inters \o{w'}=\void$, so $\bd{\o{z}}\cont \bd{\o{w'}}$. We conclude that $\bd{\o{z}}= \bd{\o{w'}}$. By the same argument, there exists a component $\o{z'}$ of $L\difference \bd{\o{w'}}=L\difference \bd{\o{z}}$ such that $\o{z}\cont \o{z'}$ and $\bd{\o{z'}}=\bd{\o{w'}}$. Since $\o{z'}\cont L\difference \bd{\o{z}}$ is connected, we obtain that $\o{z'}\cont \o{z}$, whence $\o{z}=\o{z'}$. Therefore $\o{z}$ and $\o{w'}$ are disjoint components of $L\difference \bd{\o{z}}$ with common boundary. But $\bd{\o{z}}=B_1\join B_2$ is not connected, so \Rom{4} does not hold.

\Rom{3} $\implic$ \Rom{5}. Let $C$ and $D$ be disjoint complemented connected sublocales of $L$ such that $\bd{C}\cont \bd{D}$. We want to show that $\bd{C}$ is connected. We can assume that $C\neq \void$ and $C\neq L$, otherwise $\bd{C}=\void$ is trivially connected. Notice that $\closure{C}$ and $\closure{L\difference C}$ are then non-void closed sublocales of $L$ such that $\closure{C}\join  \closure{L\difference C}=L$ and $\closure{C}\inters \closure{L\difference C}=\bd{C}$. Indeed, by \lemx\ref{boundarydiff} $\bd{L\difference C}=\bd{C}$ and then
		$$\closure{C}\inters \closure{L\difference C}=(\interior{C}\join \bd{C})\inters (\interior{L\difference C}\join \bd{L\difference C})=\bd{C}\join \bd{L\difference C}=\bd{C}.$$
		So it suffices to show that $\closure{C}$ and $\closure{L\difference C}$ are connected, because then by \Rom{3} $\closure{C}\inters \closure{L\difference C}=\bd{C}$ is connected. $\closure{C}$ is connected because closure of a connected sublocale. Consider the sublocale
		$$X:=\closure{D}\join \bigjoin\set{E\cont L\difference (\closure{C\join D})}{E \text{ component}}.$$
		This join exists because the $E$'s are all open sublocales. Moreover, by \prox\ref{3-4}, for every component $E$ of $L\difference (\closure{C\join D})$ we have $\bd{E}\neq \void$. Notice that by \prox\ref{componentbd}, \lemx\ref{boundarydiff} and \lemx\ref{bdinters}
		$$\bd{E}\cont \bd{L\difference (\closure{C\join D})}=\bd{\closure{C\join D}}\cont \bd{C\join D}\cont \bd{C}\join \bd{D}=\bd{D}.$$
		So $\closure{E}\inters \closure{D}\contain\bd{E}\inters \bd{D}\neq \void$. By \prox\ref{joinconnected}, we get that $X$ is connected. We prove that $X=\closure{L\difference C}$. Since $L$ is locally connected, $X=\closure{D}\join (L\difference{(\closure{C\join D})})$. And by \lemx\ref{boundarydiff} $\closure{L\difference C}=\interior{L\difference C}\join \bd{L\difference C}=\interior{L\difference C}\join \bd{C}$. Since $C$ and $D$ are disjoint, $D\cont L\difference C$ and thus $\closure{D}\cont \closure{L\difference C}$. Moreover $L\difference (\closure{C\join D})\cont L\difference (C\join D)\cont L\difference C\cont \closure{L\difference C}$. So $X\cont \closure{L\difference C}$. We then have that $\bd{C}\cont \bd{D}\cont \closure{D}\cont X$. And 
		$$\interior{L\difference C}\inters \left(L\difference (\closure{D}\join (L\difference (\closure{C\join D}))) \right)=\interior{L\difference C}\inters \left(L\difference \closure{D}\right) \inters \left(\closure{C\join D}\right)=\void$$
		because since $\bd{C}\cont \bd{D}$
		$$\interior{L\difference C}\inters (L\difference \closure{D})\cont (L\difference C)\inters (L\difference \closure{D})=L\difference (C\join \closure{D})=L\difference (\closure{C}\join \closure{D})\cont L\difference (\closure{C\join D}).$$
		So $\interior{L\difference C}\cont \closure{D}\join (L\difference{(\closure{C\join D})})=X$. Thus $\closure{L\difference C}\cont X$ and we conclude that $\closure{L\difference C}= X$. Since $X$ is connected, we obtain that $\closure{L\difference C}$ is connected.

\Rom{5} $\implic$ \Rom{6}. Let $C$ be a simple sublocale. Then $C$ and $L\difference C$ are disjoint connected sublocales of $L$ and by Lemma \ref{boundarydiff} we have $\bd{C}= \bd{L \difference C}$. So, by \Rom{5}, we conclude that $\bd{C}$ is connected.

\Rom{6} $\implic$ \Rom{7}. Trivial.

\Rom{7} $\implic$ \Rom{2}. Let $D$ be a component of $L\difference \cl{u}$. We prove that $L\difference D$ is connected. If $L=D$ then $L\difference D=\void$ is connected. If $L\neq D$, by Proposition \ref{3-4} $\bd{D}\neq \void$ and $\bd{D} \cont \cl{u}$. Let $\cl{l}$ and $\cl{m}$ be closed sublocales such that $L\difference D\cont \cl{l} \join \cl{m}$ and $(L\difference D) \inters \cl{l} \inters \cl{m} =\void$. Since $\cl{u} \cont L\difference D$ is connected, we can assume that $\cl{u} \cont \cl{l}$ and $\cl{u} \inters \cl{m}=\void$ and so $\bd{D} \inters \cl{m}=\void$. Consider now the closed sublocales $\cl{l} \join \closure{D}$ and $\cl{m} \difference \interior{D}= \cl{m} \difference D$. We have that $(\cl{l} \join \closure{D}) \join  (\cl{m} \difference \interior{D}) = L$ and $(\cl{l} \join \closure{D}) \inters (\cl{m} \difference \interior{D}) =(\cl{l} \inters \cl{m} \inters L\difference D ) \join (\bd{R} \inters \cl{m})=\void$. Since $L$ is connected, it follows that either $\cl{l} \join \closure{D}=\void$ or $\cl{m} \difference \interior{D}=\void$. But $D\neq \void$ and so it must be $\cl{m} \difference \interior{D}=(L\difference D) \inters \cl{m}=\void$. Hence $L \difference D$ is connected and so $D$ is a simple region. By \Rom{7}, we then conclude that $\bd{D}$ is a continuum.

\Rom{5} $\implic$ \Rom{8}. Trivial.

\Rom{8} $\implic$ \Rom{9}. We prove that if \Rom{9} does not hold then \Rom{8} does not hold. So let $A$ and $B$ be regions such that $\bd{A}\inters \bd{B}=\void$ and $A\inters B$ is not connected. We want to construct disjoint regions $D$ and $E$ such that $\bd{D}=\bd{E}$ but $\bd{D}$ is not connected. Let $C$ be a component of $A\inters B$. Since $C$ is a component of an open sublocale, $C$ is open and thus complemented. Then $C\neq A\inters B$ and so $(A\inters B)\difference \closure{C}=(A\inters B)\difference C\neq \void$. Indeed, $C$ is closed in $A\inters B$ and thus by \prox\ref{closedofsub} $C=\closure[A\inters B]{C}=\closure{C}\inters A\inters B$. Let then $D$ be a component of $L\difference\closure{C}$ such that $D\inters ((A\inters B)\difference \closure{C})\neq \void$. Such a $D$ exists thanks to \corx\ref{corolljoinint}, and $D$ is open. Since $C\inters D=\void$ and $C$ is open, by \prox\ref{closureintopen}, $C\inters \closure{D}=\void$. Let then $E$ be a component of $L\difference \closure{D}$ such that $C\inters E\neq \void$, again by \corx\ref{corolljoinint}, so that $C\cont E$ since $C$ is connected. We have that $\bd{D}=\bd{E}$. Indeed by \prox\ref{3-4} $\bd{E}\cont \closure{D}$. Moreover $D\inters E=\void$, so $D\inters \closure{E}=\void$ because $D$ is open (by \prox\ref{closureintopen}) and then $D\inters \bd{E}=\void$. So $\bd{E}\cont \bd{D}$. Then by \prox\ref{3-4} $\bd{D}\cont \closure{C}\cont \closure{E}$. And $\bd{D}\inters E=\void$ because $E\inters \closure{D}=\void$. So $\bd{D}\cont \bd{E}$ and we conclude that $\bd{D}= \bd{E}$. $D$ and $E$ are disjoint regions such that $\bd{D}= \bd{E}$, but $\bd{D}$ is not connected. Indeed $\bd{D}\cont \bd{C}$, since by \prox\ref{3-4} $\bd{D}\cont \closure{C}$ and $\bd{D}\inters C\cont \closure{D}\inters C=\void$. And $\bd{C}\cont \bd{A\inters B}\cont \bd{A}\join \bd{B}$, by \prox\ref{componentbd} and \lemx\ref{bdinters}. We have that $\bd{D}\inters \bd{A}\inters \bd{B}=\void$. Moreover $A\inters D\neq\void\neq B\inters D$ by construction of $D$, $C\cont A\inters (L\difference D)$ and $C\cont B\inters (L\difference D)$. Since $A$ and $B$ are connected, by \prox\ref{propconnormcon} $A\inters \bd{D}\neq \void$ and $B\inters \bd{D}\neq \void$. So $\bd{D}\cont \bd{A}\join \bd{B}$ exhibits a separation of $\bd{D}$, because if we had $\bd{D}\cont \bd{A}$ we would get $\void\neq A\inters \bd{D}\cont A\inters \bd{A}$, which contradicts $A$ being open, and analogously if we had $\bd{D}\cont \bd{B}$. We conclude that \Rom{8} does not hold.

\Rom{9} $\implic$ \Rom{1}. Let $X,Y\cont L$ with $X\neq \void\neq Y$ and $\cl{u},\cl{v}\cont L$ be such that $\cl{u}\inters\cl{v}=\void$ and neither $\cl{u}$ nor $\cl{v}$ separate any element of $X$ (different from 1) from any element of $Y$ (different from 1) in $L$. We prove that $\cl{u}\join \cl{v}$ does not separate $X$ and $Y$ in $L$. We can assume that $\cl{u}\neq\void\neq\cl{v}$, otherwise $\cl{u}\join \cl{v}$ does not separate $X$ and $Y$ in $L$, trivially. We can also assume that $X,Y\cont L\difference (\cl{u}\join \cl{v})=(L\difference \cl{u})\inters (L\difference \cl{v})$, otherwise again $\cl{u}\join \cl{v}$ does not separate $X$ and $Y$, trivially. Since $L\difference \cl{u}$ is open, there exists a component of $L\difference \cl{u}$ that contains an element of $X$ different from 1 and there exists a component of $L\difference \cl{u}$ that contains an element of $Y$ different from 1. Since $\cl{u}$ does not separate any element of $X$ from any element of $Y$, these components must be the same one, that we call $A$. Analogously, since $L\difference \cl{v}$ is open and $\cl{v}$ does not separate any element of $X$ from any element of $Y$, there exists a component $B$ of $L \difference \cl{v}$ that contains an element of $A\inters X$ different from 1 and an element of $A\inters Y$ different from 1. Since by \prox\ref{3-4} $\bd{A} \inters \bd{B} \cont \cl{u} \inters \cl{v}= \void$, by \Rom{9} $A \inters B$ is connected. Moreover, since $(A \inters B) \inters (\cl{u} \join \cl{v})=(A \inters B \inters \cl{u}) \join (A \inters B \inters \cl{v}) = \void$, we have $A \inters B \cont L \difference (\cl{u} \join \cl{v})$. So $A\inters B$ is a connected sublocale of $L \difference (\cl{u} \join \cl{v})$ that contains an element of $X$ and an element of $Y$ and this implies that $\cl{u} \join \cl{v}$ does not separate $X$ and $Y$ in $L$.

\Rom{9} $\implic$ \Rom{10}. Let $A$ and $B$ be open connected sublocales of $L$ such that $L= A \join B$. We need to prove that $A\inters B$ is connected.  Since $A$ is open, $\bd{A} \inters A=\void$ and so $\bd{A} \cont B$. Since $\bd{B} \inters B=\void$, it follows that $\bd{A} \inters \bd{B}=\void$. Then, by \Rom{9}, $A \inters B$ is connected.

\Rom{10} $\implic$ \Rom{9}.
Let $A$ and $B$ be regions such that $\bd{A}\inters \bd{B}=\void$. If $A\inters B=\void$ then it is connected and so we can assume $A\inters B\neq \void$. To prove that $A\inters B$ is connected using open unicoherence, we construct $A\st$ and $B\st$ open connected sublocales of $L$ such that $A \cont A\st$, $B\cont B\st$, $A\st \inters B\st=A \inters B$ and $A\st \join B \st=L$. We define
$$A\st=A \join \left(\bigjoin \set{D\cont L\difference \closure{B}}{ D \text{ component}, D\inters A\neq \void}\right)\v[-2]$$
$$B\st=B \join \left(\bigjoin \set{D\cont L\difference \closure{A}}{ D \text{ component}, D\inters B\neq \void}\right)$$
Clearly $A\cont A\st$ and $B\cont B\st$. Moreover, $A\st$ and $B\st$ are open (since joins of open sublocales) and they are connected by Proposition \ref{joinconnected}. To prove that $A\st \inters B\st=A \inters B$, we first observe that 
$A\st \inters B\st$ is the join of the following sublocales:
\begin{itemize}
\item $A \inters B$;
\item $B \inters (\bigjoin \set{D\cont L\difference \closure{B}}{ D \text{ component}, D\inters A\neq \void})$;
\item $A\inters (\bigjoin \set{D\cont L\difference \closure{A}}{ D \text{ component}, D\inters B\neq \void}) $;
\item $(\bigjoin \set{D\cont L\difference \closure{B}}{ D \text{ component}, D\inters A\neq \void}) \inters$\\
$ (\bigjoin \set{D\cont L\difference \closure{A}}{D \text{ component}, D\inters B\neq \void})$.
\end{itemize}
Since the second and the third sublocales are void and the components involved are open so we can distribute the joins, it suffices to show that given $D\cont L\difference \closure{B}$ component such that $D\inters A\neq \void$ and $D'\cont L\difference \closure{A}$ component such that $D'\inters B\neq \void$, we have $D\inters D'=\void$. Since $D\inters A \neq \void$, $D$ cannot be contained in $L \difference (\closure{A} \join \closure{B})$ and hence $D$ is not contained in  $D'$ and so $D \inters (L\difference D')\neq \void$. Analogously $D'$ is not contained in $D$ and so $D' \inters (L\difference D)\neq \void$. Moreover, since $A\inters B\neq \void$, we have that $D \join D' \cont L \difference (\closure{A} \inters \closure{B})$ is not the entire $L$. By Proposition \ref{compdiffissimple}, $L\difference D$ is connected and so, since $(L\difference D) \inters D'\neq \void$ and $(L\difference D) \inters (L\difference D')=L \difference (D \join D') \neq \void$, thanks to Proposition \ref{propconnormcon} it must be $(L\difference D) \inters \bd{D'} \neq \void$.  Observe now that $\bd{D}\cont \bd{B}$. Indeed $D\neq L$ thanks to the uniqueness of the pseudo-complement, because $\closure{B}\neq \void$. By \prox\ref{3-4}, $\bd{D}\cont \closure{B}=\interior{B}\join \bd{B}$. And $\bd{D}\inters \interior{B}=\void$, because otherwise we would get an element $x\in \bd{D}\inters \interior{B}$ and thus, by \lemx\ref{neighinterior} and \lemx\ref{neighboundary}, an open sublocale $U_x$ containing $x$ such that $U_x\cont B$ and $U_x\inters D\neq \void$. The latter is a contradiction, since $B\inters D=\void$. So $\bd{D}\cont \bd{B}$, and analogously $\bd{D'}\cont \bd{A}$. This implies that $\bd{D} \inters \bd{D'}= \void$ as $\bd{A} \inters \bd{B}= \void$. Since $L\difference D'$ is connected by Proposition\ref{compdiffissimple}, by Lemma \ref{lemmasimplebd} $\bd{D'}$ is normally connected. Since $\bd{D} \inters \bd{D'}=\void$ and $(L\difference D) \inters \bd{D'} \neq \void$, by Proposition \ref{propconnormcon} it must be $D\inters \bd{D'}=\void$. Since $D$ is connected and $D \inters (L\difference D') \neq \void$, by Proposition \ref{propconnormcon}, we must have $D\inters D'=\void$. Hence $A\st \inters B\st=A \inters B$.

It remains to prove that $A \st \join B\st =L$. We first observe that $\bd{A} \cont L\difference B$ since $\bd{A} \inters \bd{B}= \void$ and hence $\bd{A} \inters (L\difference(A \join B))\cont L\difference \closure{B}$. So $\bd{A}\inters   (L\difference(A \join B))$ is covered by components of $L\difference \closure{B}$. Since the components of $L\difference \closure{B}$ are open sublocales, by Lemma \ref{neighboundary}, the ones that intersect $\bd{A}$ must intersect $A$ as well and so $\bd{A} \inters (L\difference(A \join B))\cont A\st$. On the other hand, $(L \difference \bd{A}) \inters (L \difference (A\join B)) \cont L\difference \closure{A}$ and so $(L \difference \bd{A}) \inters (L \difference (A\join B))$ is covered by the components of $L\difference \closure{B}$. Let $D$ be a component of $L\difference \closure{A}$. If $D\inters B\neq\void$, then $D \cont B\st$. If $D\inters B =\void$, then by Proposition \ref{closureintopen}, $D \inters \closure{B}=\void$ and so $D \cont (L\difference \closure{A}) \inters (L\difference \closure{B})$. Then by \corx\ref{corolljoinint}, since $D$ is connected, $D\cont E$ with $E$ component of $L\difference \closure{B}$. By the argument above, $\bd{D}\cont \bd{A}$ and $\bd{E}\cont \bd{B}$. Since $\bd{A} \inters \bd{B}=\void$, $D$ cannot be equal to $E$ and hence $E$ is not contained in $L\difference \closure{A}$ (because $D$ is a component of $L\difference \closure{A}$). This implies that $E\inters A\neq\void$ and so $D\cont A\st$. We have then shown that $D\cont A\st \join B\st$ and so $(L \difference \bd{A}) \inters (L \difference (A\join B))\cont A\st \join B\st$. Since we had already proved that $\bd{A} \inters (L \difference (A\join B))\cont A\st \join B\st$, we have shown that $L \difference (A\join B) \cont A\st \join B\st$. But $A \join B \cont A\st \join B\st$ by construction and so $A \st \join B\st =L$.
\end{proof}

\begin{teor}\label{thmstrloccon}
    Let $L$ be a connected and strongly locally connected locale. The following properties are equivalent:
    \begin{enumerate}[itemsep=2.5mm]
\item[\Romplus{1}] Whenever $x,y\in L$ with $x\neq 1\neq y$ and $\cl{u},\cl{v}\cont L$ are such that $\cl{u}\inters\cl{v}=\void$ and neither $\cl{u}$ nor $\cl{v}$ separate $x$ and $y$ in $L$, we have that $\cl{u}\join \cl{v}$ does not separate $x$ and $y$ in $L$;
\item[\Rom{1}] Whenever $X,Y\cont L$ with $X\neq \void\neq Y$ and $\cl{u},\cl{v}\cont L$ are such that $\cl{u}\inters\cl{v}=\void$ and neither $\cl{u}$ nor $\cl{v}$ separate any element of $X$ (different from 1) from any element of $Y$ (different from 1) in $L$, we have that $\cl{u}\join \cl{v}$ does not separate $X$ and $Y$ in $L$;
\item[\Romprime{1}](Phragmen-Brouwer Property) Whenever $\cl{u},\cl{v}\cont L$ are such that $\cl{u}\inters\cl{v}=\void$ and neither $\cl{u}$ nor $\cl{v}$ separate $L$, we have that $\cl{u}\join \cl{v}$ does not separate $L$.
    \end{enumerate}
    Moreover, \Romplus{1} $\implic$ \Rom{1} $\implic$ \Romprime{1} actually holds for any connected and locally connected locale $L$, not necessarily strongly locally connected.
\end{teor}
\begin{proof}
\Romplus{1} $\implic$ \Rom{1}. Let $X,Y\cont L$ with $X\neq \void\neq Y$ and let $\cl{u},\cl{v}\cont L$ be closed sublocales such that $\cl{u}\inters\cl{v}=\void$ and neither $\cl{u}$ nor $\cl{v}$ separate any element of $X$ (different from 1) from any element of $Y$ (different from 1) in $L$. Let $x\in X$ and $y\in Y$ with $x\neq 1$ and $y\neq 1$. Since neither $\cl{u}$ nor $\cl{v}$ separate $x$ and $y$, by \Romplus{1} $\cl{u}\join\cl{v}$ does not separate $x$ and $y$. Whence, $\cl{u}\join\cl{v}$ does not separate $X$ and $Y$.

 \Rom{1} $\implic$ \Romprime{1}. Let $\cl{u},\cl{v}\cont L$ be such that $\cl{u}\inters\cl{v}=\void$ and $\cl{u}\join \cl{v}$ separates $L$. We prove that either $\cl{u}$ or $\cl{v}$ separate $L$. Let $L\difference (\cl{u} \join \cl{v}) \cont \cl{l} \join \cl{m}$ be a separation of $L\difference (\cl{u} \join \cl{v})$. Let then $X=(L\difference (\cl{u} \join \cl{v})) \inters \cl{l}$ and $Y=(L\difference (\cl{u} \join \cl{v})) \inters \cl{m}$. Since $\cl{l}$ and $\cl{m}$ give a separation of $L\difference (\cl{u} \join \cl{v})$, we have $X\neq \void$ and $Y\neq \void$. Then $\cl{u}\inters\cl{v}$ separates $X$ from $Y$ by construction and so by \Rom{1} there exist $x\in X$ and $Y\in Y$ with $x\neq 1$ and $y\neq 1$ such that either $\cl{u}$ separates $x$ and $y$ or $\cl{v}$ does so. Whence either $\cl{u}$ separates $L$ or $\cl{v}$ does so.
    
    \Rom{4} $\implic$ \Romplus{1}, so that \Romplus{1} $\iff$ \Rom{1}. Now that $L$ is strongly locally connected, we can improve the proof that \Rom{4} $\implic$ \Rom{1} of \thex\ref{mainthm}. We prove that if \Romplus{1} does not hold, then \Rom{4} does not hold. So let $x,y\in L$ with $x\neq 1 \neq y$ and $\cl{u},\cl{v}\cont L$ be sublocales such that $\cl{u}\inters \cl{v}=\void$ and neither $\cl{u}$ nor $\cl{v}$ separate $x$ and $y$ in $L$ but $\cl{u}\join \cl{v}$ separates $x$ and $y$ in $L$. In particular we have that $x,y\in L\difference (\cl{u}\join \cl{v})$. Since $L$ is strongly locally connected, by \prox\ref{strlocconisloccon}, there exists a component $\o{w}$ of $L\difference (\cl{u}\join \cl{v})$ such that $x\in \o{w}$. The rest of the proof proceeds as the proof of \Rom{4} $\implic$ \Rom{1} of \thex\ref{mainthm}. $\bd{\o{w}}$ separates $x$ and $y$ in $L$. Whereas $F_1$ and $F_2$ do not separate $x$ and $y$ in $L$, otherwise $\cl{u}$ or $\cl{v}$ would separate $x$ and $y$ in $L$. We can then take $\o{z}$ to be the component of $L\difference \bd{\o{w}}$ such that $y\in \o{z}$, by \prox\ref{strlocconisloccon}, using again that $L$ is strongly locally connected. We then still use $x$ and $y$ rather than $X\inters \o{w}$ and $Y\inters \cl{b}$, to find the separation $\bd{\o{z}}=B_1\join B_2$. We 
		then follow the proof of \Rom{4} $\implic$ \Rom{1} to the end, arriving to the conclusion that \Rom{4} does not hold.

    \Romprime{1} $\implic$ \Rom{1}. We prove that if \Rom{1} does not hold then \Romprime{1} does not hold. We first repeat the whole proof of \Rom{4} $\implic$ \Rom{1} of \thex\ref{mainthm}. So, starting from the assumption that \Rom{1} does not hold, we build the disjoint components $\o{z}$ and $\o{w'}$ of $L\difference \bd{\o{z}}$ with common boundary. Remember that we have a separation $\bd{\o{z}}=B_1\join B_2$, and that $B_1$ and $B_2$ are closed. We define
		$$B_1':=B_1\join \left(\bigjoin\set{\o{l} \text{ component of } L\difference \bd{\o{z}}}{\bd{\o{l}}\cont B_1}\right)$$
		$$B_2':=B_2\join \left(\bigjoin\set{\o{m} \text{ component of } L\difference \bd{\o{z}}}{\bd{\o{m}}\cont B_2}\right)$$
		Then $B_1'$ is closed, because by \lemx\ref{bdinters} and \lemx\ref{bdinfinitejoin}
		$$\bd{B_1'}\cont \bd{B_1}\join \bd{\bigjoin\fami{\o{l}}{l}}\cont B_1\join \left(\closure{\bigjoin\fami{\bd{\o{l}}}{l}}\right)\cont B_1\join \closure{B_1}=B_1\cont B_1'.$$
		Analogously, $B_2'$ is closed. Notice then that $B_1'\inters B_2'=\void$, thanks to \lemx\ref{lemmajoinopens}. Indeed, if we had that $\o{l}\inters \o{m}\neq \void$ for some components $\o{l}$ and $\o{m}$ as above, we would get $\o{l}=\o{m}$, contradicting that $B_1\inters B_2=\void$, since by \prox\ref{3-4} $\bd{\o{l}}\neq \void$.
		
		\noindent We prove that neither $B_1'$ nor $B_2'$ separate $L$, whereas $B_1'\join B_2'$ separates $L$. Consider $C$ a component of $L\difference B_1'$. Consider then an arbitrary component $C'$ of $L\difference \bd{\o{z}}$. By \prox\ref{3-4}, $\bd{C'}\neq \void$ and $\bd{C'}\cont B_1\join B_2$. If $\bd{C'}\inters B_2=\void$, then $\bd{C'}\cont B_1$ and so $C'\cont B_1'$. If $\bd{C'}\inters B_1=\void$, then $\bd{C'}\cont B_2$ and so $C'\cont B_2'\cont L\difference B_1'$. If $\bd{C'}\inters B_1\neq\void\neq\bd{C'}\inters B_2$, then $C'\inters B_1'=\void$ and so again $C'\cont L\difference B_1'$. Notice that it cannot happen that $C\cont C'$. Indeed, this would mean that either $C\cont C'\cont B_1'$, which contradicts the fact that $C\cont L\difference B_1'$, or $C\cont C'\cont L\difference B_1'$. In this second case, we would get that $C=C'$ because $C$ is a component. Whence by \prox\ref{3-4} $\bd{C}\cont B_1'\inters \bd{\o{z}}=B_1$ and we would have $C\cont B_1'$, contradicting again that $C\cont L\difference B_1'$. We can then conclude that $C\inters B_2\neq \void$. Indeed, if we had $C\inters B_2=\void$ then we would have $C\cont L\difference B_1$ and $C\cont L\difference B_2$, so $C\cont (L\difference B_1)\inters (L\difference B_2)=L\difference (B_1\join B_2)=L\difference \bd{\o{z}}$. By \corx\ref{corolljoinint}, there would then exists a component $C'$ of $L\difference \bd{\o{z}}$ such that $C'\inters C\neq\void$. And since $C$ is connected, $C\cont C'$, which cannot happen by the argument above. So every component $C$ of $L\difference B_1'$ needs to intersect $B_2$. Notice that $\o{w'}\inters B_1'=\void$ and so $\o{w'}\cont L\difference B_1'$, because $\o{w'}$ is an open component and its boundary is not contained in $B_1$. Notice then that $B_2\cont \bd{\o{w'}}\cont \closure{\o{w'}}$ and $B_2\cont L\difference B_1'$. So by \prox\ref{closedofsub} $B_2\cont \closure[L\difference B_1']{\o{w'}}$. But $\closure[L\difference B_1']{\o{w'}}$ is connected in $L\difference B_1'$. And so every component $C$ of $L\difference B_1'$ needs to intersect the same connected sublocale $\closure[L\difference B_1']{\o{w'}}$ of $L\difference B_1'$. We conclude that $L\difference B_1'$ is connected. Analogously, $L\difference B_2'$ is connected. But $L\difference (B_1'\join B_2')$ is not connected, since $\o{w'}$ and $\o{z}$ are disjoint components of $L\difference (B_1'\join B_2')$. Indeed we said above that $\o{w'}\cont L\difference B_1'$. Analogously, $\o{w'}\cont L\difference B_2'$ and also $\o{z}\cont (L\difference B_1')\inters (L\difference B_2')=L\difference (B_1'\join B_2')$. So \Rom{1} does not hold.
\end{proof}

\begin{rem}
    For topological spaces, there is no difference between \Romplus{1} and \Rom{1}, as we can always consider the singleton subspace consisting of a point. But for sublocales, the property that works best seems to be $\Rom{1}$. Analogously with the properties $(N^+)$ and $(N)$ below.
\end{rem}

\begin{teor}\label{thmnormal}
    Let $L$ be a connected and locally connected locale. Consider the following properties:
    \begin{enumerate}[itemsep=2.5mm]
\item[$(N^+)$] If $\cl{u},\cl{v}\cont L$ are disjoint and $x\in \cl{u}$ and $y\in \cl{v}$ are such that $x\neq 1 \neq y$, then there exists a continuum $\cl{w}\cont L\difference (\cl{u}\join \cl{v})=\o{u\meet v}$ which separates $x$ and $y$ in $L$;
\item[$(N)$] If $\cl{u},\cl{v}\cont L$ are disjoint and $X\cont\cl{u}$ and $Y\cont \cl{v}$ are such that $X\neq\void \neq Y$, then there exists a continuum $\cl{w}\cont L\difference (\cl{u}\join \cl{v})=\o{u\meet v}$ which separates an element of $X$ (different from 1) from an element of $Y$ (different from 1) in $L$.
    \end{enumerate}
Then $(N^+)\implic(N)\implic \Rom{3}$.

Moreover, if $L$ is also normal then $\Rom{2}\implic (N)$, so that $(N)$ is equivalent to unicoherence.

If $L$ is also normal and strongly locally connected, then $\Rom{2}\implic (N^+)$, so that $(N^+)\iff (N)$ is equivalent to unicoherence (and equivalent to \Romplus{1}).
\end{teor}

\begin{proof}
$(N^+)$ $\implic$ (N). Let $\cl{u},\cl{v}\cont L$ be disjoint and  let $X\cont\cl{u}$ and $Y\cont \cl{v}$ with $X\neq\void \neq Y$. Let then $x$ be an element of $X$ with $x\neq 1$ and $y$ be an element of $Y$ with $y\neq 1$. By $(N^+)$, there exists a continuum $\cl{w}\cont L\difference (\cl{u}\join \cl{v})=\o{u\meet v}$ which separates $x$ and $y$ in $L$. So $\cl{w}\cont L\difference (\cl{u}\join \cl{v})=\o{u\meet v}$ is a continuum which separates an element of $X$ (different from 1) from an element of $Y$ (different from 1) in $L$.

(N) $\implic$ \Rom{3}. We prove that if \Rom{3} does not hold then (N) does not hold. Let $\cl{u}$ and $\cl{v}$ be continua such that $L=\cl{u} \join \cl{v}$ but $\cl{u} \inters \cl{v}$ is not a continuum. Since $\cl{u} \inters \cl{v}$ is clearly closed and it cannot be $\cl{u} \inters \cl{v}= \void$ because $L$ is connected, it must be $\cl{u} \inters \cl{v}$ not connected. By Lemma \ref{closedsubconnected}, there exist $\cl{a}$ and $\cl{b}$ non-void such that $\cl{u} \inters \cl{v}= \cl{a} \join \cl{b}$, $\cl{a} \inters \cl{b}=\void$, $\cl{a} \inters \cl{u} \inters \cl{v} \neq \void$ and $\cl{b} \inters \cl{u} \inters \cl{v} \neq \void$. Let $X=\cl{a} \inters \cl{u} \inters \cl{v}\cont \cl{a}$ and $Y=\cl{b} \inters \cl{u} \inters \cl{v}\cont \cl{b}$ and consider a continuum $\cl{w}\cont L \difference (\cl{a} \join \cl{b})=L\difference (\cl{u} \inters \cl{v})$ that separates $L$. We show that $\cl{w}$ cannot separate any element $x$ of $X$ (different from 1) from any element $y$ of $Y$ (different from 1). Consider such arbitrary $x$ and $y$. Notice that $\cl{w}\cont L=\cl{u}\join \cl{v}$ and $\cl{w}\inters \cl{u}\inters \cl{v}\cont (L\difference (\cl{u}\inters \cl{v}))\inters \cl{u}\inters \cl{v}=\void$. Since $\cl{w}$ is connected, either $\cl{w}\inters \cl{u}=\void$ or $\cl{w}\inters \cl{v}=\void$. Without loss of generality, $\cl{w}\inters \cl{u}=\void$. Then $\cl{u}\cont L\difference \cl{w}$. Therefore $x,y\in \cl{u}$ belong to the same connected sublocale of $L\difference \cl{w}$, and so cannot be separated by $\cl{w}$. We conclude that (N) does not hold.

\Rom{2} $\implic$ (N), assuming $L$ is normal. Let $\cl{a}, \cl{b}\cont L$ be disjoint closed sublocales and let $X\cont \cl{a}$ and $Y\cont \cl{b}$ be such that $X\neq \void \neq Y$. We need to prove that there exists a continuum contained in $L\difference (\cl{a} \join \cl{b})$ which separates an element of $X$ from an element of $Y$. Since $L$ is normal, there exist $\o{u},\o{v} \cont L$ disjoint open sublocales such that $\cl{a} \cont \o{u}$ and $\cl{b} \cont \o{v}$. Since $X\inters \o{u}=X\neq \void$, by Corollary \ref{corolljoinint}, there exists a component $C$ of $\o{u}$ such that $X\inters C\neq \void$. Let $x$ be an element of $X\inters C$ such that $x\neq 1$. We now observe that $Y \inters (L\difference \closure{C}) \neq \void$. Indeed, since $\o{v} \inters \o{u}=\void$ and $\o{v}$ is open, by Proposition \ref{closureintopen}, it must be $\o{v} \inters \closure{\o{u}}=\void$ and so $Y\cont \o{v} \cont L \difference \closure{\o{u}} \cont L \difference \closure{C}$. Then, by Corollary \ref{corolljoinint}, there exists a component $D$ of $L \difference \closure{C}$ such that $D \inters Y \neq \void$. Let $y$ be an element of $D \inters Y$ such that $y\neq 1$. Since $C$ is connected, also $\closure{C}$ is connected and hence it is a continuum. So, by \Rom{2}, we have that $\bd{D}$ is a continuum. We observe that $\closure{D} \neq L$ because $\closure{D} \inters C=\void$ by Proposition \ref{closureintopen} since $C\inters D=\void$. Moreover, $\bd{D} \cont L\difference (\cl{a} \join \cl{b})$. Indeed, since $\bd{D} \cont \closure{C}$ by Proposition \ref{3-4} and $\bd{D} \inters C \cont \closure{D} \inters C=\void$, we have $\bd{D} \cont \bd{C}$. But $\bd{C} \cont L \difference \o{u} \cont L\difference \cl{a}$ by Proposition \ref{3-4} and $\bd{C} \cont \closure{C} \cont \closure{\o{u}} \cont L \difference \o{v} \cont L \difference \cl{b}$. So $\bd{D} \cont (L\difference \cl{a}) \inters (L\difference \cl{b})=L \difference (\cl{a} \join \cl{b})$. So $\bd{D}$ is a continuum of $L \difference (\cl{a} \join \cl{b})$.  By Proposition \ref{3-3}, we then have a separation $L\difference \bd{D}= D \join (L\difference \closure{D})$ and as a consequence also a separation $L\difference \bd{D}\cont \closure{D} \join \closure{(L\difference \closure{D})}$. But $x\in C$ and $C\cont L \difference \closure{D} \cont \closure{L \difference \closure{D}}$ because $C\inters \closure{D}=\void$ since $C\inters D=\void$ and $C\inters \bd{D} \cont D \inters \bd{D}=\void$. Hence $x\in \closure{L\difference{\closure{D}}}$. And $y\in D \cont \closure{D}$. So $\bd{D}$ is a continuum contained in $L\difference (\cl{a} \join \cl{b})$ which separates an element of $X$ from an element of $Y$.

\Rom{2}$\implic (N^+)$, assuming that $L$ is normal and strongly locally connected. Now that $L$ is strongly locally connected we can improve the proof that \Rom{2}$\implic$ (N). Let $\cl{a}, \cl{b}\cont L$ be disjoint closed sublocales and let $x\in\cl{a}$ and $y\in\cl{b}$ be such that $x\neq 1\neq y$. We need to prove that there exists a continuum contained in $L\difference (\cl{a} \join \cl{b})$ which separates an $x$ and $y$. Since $L$ is normal, there exist $\o{u},\o{v} \cont L$ disjoint open sublocales such that $\cl{a} \cont \o{u}$ and $\cl{b} \cont \o{v}$. Since $x\in \o{u}$, by Proposition \ref{strlocconisloccon}, there exists a component $C$ of $\o{u}$ such that $x\in C$. We now observe that $Y \inters (L\difference \closure{C}) \neq \void$. Indeed, since $\o{v} \inters \o{u}=\void$ and $\o{v}$ is open, by Proposition \ref{closureintopen}, it must be $\o{v} \inters \closure{\o{u}}=\void$ and so $y\in  \o{v} \cont L \difference \closure{\o{u}} \cont L \difference \closure{C}$. Then, by Proposition \ref{strlocconisloccon}, there exists a component $D$ of $L \difference \closure{C}$ such that $y\in D$. The rest of the proof proceeds as the proof that \Rom{2}$\implic$ (N) and we conclude that $\bd{D}$ is a continuum contained in $L\difference (\cl{a} \join \cl{b})$ which separates $x$ and $y$.
\end{proof}



\end{document}